\documentclass[a4paper,reqno]{amsart}

\usepackage{amsmath,amssymb,amsfonts,amsthm,amscd}
\usepackage{mathrsfs}
\usepackage{bm}
\usepackage{latexsym}
\usepackage{graphicx}
\usepackage[all,cmtip]{xy}
\usepackage{makecell}
\usepackage{multirow}
\usepackage{appendix}
\usepackage[colorlinks=true,linkcolor=blue,citecolor=blue,urlcolor=blue]{hyperref}

\numberwithin{equation}{section}

\allowdisplaybreaks

\def\cD{\mathcal{D}}

\def\cF{\mathcal{F}}
\def\cG{\mathcal{G}}
\def\cH{\mathcal{H}}

\def\cO{\mathcal{O}}
\def\cT{\mathcal{T}}

\def\bC{\mathbb{C}}

\def\bQ{\mathbb{Q}}

\newcommand{\bbm}{\begin{bmatrix}}
\newcommand{\ebm}{\end{bmatrix}}

\theoremstyle{plain}
\newtheorem{theorem}{Theorem}[section]
\newtheorem{conjecture}[theorem]{Conjecture}
\newtheorem{lemma}[theorem]{Lemma}
\newtheorem{proposition}[theorem]{Proposition}
\newtheorem{corollary}[theorem]{Corollary}

\theoremstyle{definition}
\newtheorem{definition}[theorem]{Definition}
\newtheorem{example}[theorem]{Example}

\theoremstyle{remark}
\newtheorem{remark}[theorem]{Remark}

\title[Numerical Reduction for Adjoint Singularities]
{Numerical Reduction and Sharp Thresholds for Adjoint Singularities of Foliated Surfaces}
\author{Shi Xu}
\address{Yau Mathematical Sciences Center, Tsinghua University, Beijing 100084, China}
\email{shixumath@163.com}



\date{}

\begin{document}

\begin{abstract}
Let \((X,\mathcal F)\) be a foliated surface over the complex numbers.  We
study the variation of \(\epsilon\)-adjoint singularities associated with the
adjoint divisor
\[
K_{\mathcal F}+\epsilon K_X,\qquad \epsilon>0.
\]

Using a numerical reduction procedure for negative definite exceptional
configurations, we classify \(\epsilon\)-adjoint log canonical singularities for
\(0<\epsilon<1/3\).  The reduction detects negative vertices, peels off the
special chains generated by them, and reduces the classification to the residual
intersections left after peeling.  In this form, the first stability threshold is
\(\epsilon=1/5\): for \(0<\epsilon<1/5\), every \(\epsilon\)-adjoint log canonical
singularity is foliated log canonical, while at \(\epsilon=1/5\) a new boundary
configuration appears.

Imposing the stronger \(\epsilon\)-adjoint canonical condition gives a second
classification below \(\epsilon=1/4\).  The same residual mechanism detects the
wall \(1/4\), which gives the sharp canonical-to-log-canonical stability
interval.  Both thresholds are sharp and are realized by explicit examples.

As an application, we describe the negative part of the Zariski decomposition of
\(K_{\mathcal F}+\epsilon K_X\) below the wall \(1/4\), and obtain the corresponding
stability range for the adjoint minimal model program.
\end{abstract}

\maketitle 

\setcounter{tocdepth}{2}  
\tableofcontents

\section{Introduction}

Throughout this paper, we work over the complex numbers \(\mathbb C\).

A recurring problem in surface birational geometry is to understand how
singularity conditions vary under perturbations of canonical divisors.  On an
ordinary surface, such questions are controlled by the intersection theory of
negative definite exceptional configurations.  For a foliated surface
\((X,\mathcal F)\), there are two canonical classes in play: the foliated
canonical class \(K_{\mathcal F}\), which governs the intrinsic birational
geometry of the foliation, and the ambient canonical class \(K_X\), which
governs the geometry of the underlying surface.  The adjoint divisors
\[
        K_{\mathcal F}+\epsilon K_X,\qquad \epsilon>0,
\]
therefore provide a natural way to study how the intrinsic and ambient
geometries interact.

Adjoint foliated structures were introduced by Pereira and Svaldi \cite{PS19} and were
further developed by Spicer and Svaldi for rank-one foliations on surfaces \cite{SS23}, and by
Cascini--Han--Liu--Meng--Spicer--Svaldi--Xie for algebraically integrable
foliations in higher dimensions \cite{CHL+24}.  From this point of view,
\(K_{\mathcal F}+\epsilon K_X\) is a small ambient perturbation of
\(K_{\mathcal F}\).  Such perturbations are not merely formal.  They arise
naturally in the birational theory of foliations, where the intrinsic canonical
class may exhibit pathologies, for example the possible failure of finite
generation of the canonical ring even when \(K_{\mathcal F}\) is big and nef \cite{MMcQ08},
and the failure of effective birationality for rank-one foliations \cite{Lu25}.

Once \(K_{\mathcal F}\) is replaced by \(K_{\mathcal F}+\epsilon K_X\), the
discrepancy data also change.  This leads to a basic stability problem: for
sufficiently small \(\epsilon>0\), do the singularities defined by the adjoint
divisor recover the usual singularities of the foliation, or do genuinely new
singularities already appear?

The answer is subtle.  Canonicity is not stable under arbitrarily small adjoint
perturbations: an \(\epsilon\)-adjoint canonical singularity need not be
canonical as a foliated singularity, no matter how small \(\epsilon>0\) is; see Example~\ref{ex:epsilon-ad-can-NOT-can}.
Thus the relevant question is not whether all usual singularity classes are
preserved, but rather which weaker stability statements survive and where the
first numerical walls occur.

This stability problem may be viewed as a local form of a general issue in
adjoint birational geometry: a small ambient perturbation may improve positivity,
but it also changes the discrepancy function.  Thus one needs an effective way
to measure how much of the usual foliated singularity theory survives under
\[
K_{\mathcal F}\longmapsto K_{\mathcal F}+\epsilon K_X .
\]
On surfaces, we show that this problem is governed by a numerical reduction on
negative definite exceptional configurations.

The main results of this paper identify two sharp walls.  The first one occurs
at
\[
\epsilon=\frac15.
\]
For \(0<\epsilon<1/5\), every \(\epsilon\)-adjoint log canonical singularity is
foliated log canonical.  
At \(\epsilon=1/5\), a new boundary configuration enters the admissible region,
and this stability statement fails.
The second wall occurs
at
\[
\epsilon=\frac14.
\]
If one imposes the stronger \(\epsilon\)-adjoint canonical condition, then
\(0<\epsilon<1/4\) is the sharp interval in which adjoint canonicity still
forces foliated log canonicity.

The classification leading to these walls should not be viewed as a direct
enumeration of dual graphs.  The main point is a numerical reduction procedure
for negative definite configurations.  Let \(D\) be a divisor on a smooth
surface and let \(E=\sum_i E_i\) be a negative definite configuration.  We
define the numerical projection \(M(D,E)\), supported on \(E\), by
\[
M(D,E)\cdot E_i=D\cdot E_i\quad\text{for all }i.
\]
In the situations considered here, the relevant discrepancy condition is
equivalent to the effectivity, or strict effectivity, of this projection.  The
reduction then proceeds by detecting components with negative \(D\)-intersection,
extracting the special chains forced by them, and peeling off their numerical
contributions.  What remains is a residual numerical problem on the unpeeled
part of the configuration.

Thus the walls are not found by a search through graphs.  They are the first
values of the parameter \(\epsilon\) at which one of the residual intersections
left after peeling becomes zero.  
In the \(\epsilon\)-adjoint log canonical case, the first such residual
intersection vanishes at \(\epsilon=1/5\).
In the
adjoint canonical case, the stronger effectivity condition filters the log
canonical list further and produces the wall \(\epsilon=1/4\).

The reduction procedure used in the paper is intersection-theoretic in nature.
It is formulated for negative definite configurations on smooth surfaces and
does not depend on the special features of the adjoint divisor until the local
numerical input is inserted.  
In this sense, the adjoint foliated problem provides a sharp instance of this
surface-theoretic mechanism: singularity stability is reduced to the behavior of
numerical projections under peeling, and 
the critical constants are detected by the first vanishing residual intersections.

We first state the log canonical classification in a compressed form.  
The detailed list, including the precise boundary configuration and the determinant data of the attached chains, 
is given in Theorem~\ref{thm:adjoint-lc-classification}.  
The compressed statement emphasizes the structural outcome of the reduction:
below the first wall the admissible configurations are exactly the usual foliated
log canonical ones, while at the wall a new adjoint boundary type enters the list.

\begin{theorem}[Adjoint log canonical classification]\label{thm:intro-main-lc}
Fix \(\epsilon\in(0,1/3)\).  Let \((Y,\cG,p)\) be a germ of a foliated surface
such that \(p\) is \(\epsilon\)-adjoint log canonical.  Let
\[
\pi:(X,\cF)\to(Y,\cG)
\]
be the minimal resolution over $(Y,\cG,p)$, and write \(E\) for the exceptional divisor.  Then
\(E\) belongs to one of the following structural types:
\begin{enumerate}
\item[\rm(I)] foliated canonical configurations;
\item[\rm(II)] foliated log canonical non-canonical configurations;
\item[\rm(III)] boundary configurations.
\end{enumerate}

Types {\rm (I)} and {\rm (II)} are the usual foliated log canonical types, while
Type {\rm (III)} is new in the adjoint setting and is not foliated log canonical.
It is a boundary configuration centered at a smooth rational \(\cF\)-invariant
curve \(C\) with \({\rm Z}(\cF,C)=3\).  It first occurs at
\(\epsilon=1/5\), and cannot occur for \(\epsilon<1/5\).
\end{theorem}

The precise form of the boundary configuration, as well as the detailed
classification, is given in Theorem~\ref{thm:adjoint-lc-classification}.  Its
proof is carried out in Section~\ref{sec:adjoint-lc-reduction}, using the
numerical reduction developed in Section~\ref{sec:numerical-reduction}.

\subsection{The log canonical wall \(\epsilon=1/5\)}

Theorem~\ref{thm:intro-main-lc} shows that the first new configuration in the
\(\epsilon\)-adjoint log canonical setting appears at \(\epsilon=1/5\).  Below
this value, adjoint log canonicity forces ordinary foliated log canonicity.

\begin{proposition}\label{prop:intro-lc-stability}
Fix \(\epsilon\in(0,1/5)\).  Suppose that \(p\) is an
\(\epsilon\)-adjoint log canonical singularity of \((X,\cF)\).  Then \(p\) is
log canonical both as a foliated singularity and as a surface singularity.
Moreover, if \(p\) is \(\epsilon\)-adjoint klt, then \(p\) is klt as a
singularity of \(X\).
\end{proposition}

This is proved as a consequence of the detailed classification in
Theorem~\ref{thm:adjoint-lc-classification}, together with the classification of
log canonical surface singularities recalled in Theorem~\ref{thm:lc-surface}.
The threshold \(\epsilon=1/5\) is sharp: at this value, an \(\epsilon\)-adjoint log
canonical singularity may fail to be log canonical as a foliated singularity; see
Example~\ref{ex:e=1/4}.  The numerical origin of this wall is explained in
Remark~\ref{rem:lc-wall-origin}.

\begin{remark}
Proposition~\ref{prop:intro-lc-stability} extends previous results:
\cite[Lemma~2.19]{SS23} treats the case where the underlying surface is smooth,
and \cite[Lemma~2.13]{Vas25} considers the case where the underlying surface
singularity is klt.
\end{remark}

\subsection{The canonical wall \(\epsilon=1/4\)}

The adjoint canonical condition imposes a stronger filter on the log canonical
list.  The boundary configurations which enter at the log canonical wall
\(\epsilon=1/5\) are ruled out under the canonical condition as long as
\(\epsilon<1/4\).  Thus \(\epsilon=1/4\) is the first wall in the adjoint
canonical setting.

This gives a second, smaller classification.  It also provides the local input
for the structure of the adjoint negative part and for the adjoint minimal model
program.
\begin{theorem}[Adjoint canonical classification below \(1/4\)]
\label{thm:intro-main-can}
Fix \(\epsilon\in(0,1/4)\).  Let \((Y,\cG,p)\) be a germ of a foliated surface
such that \(p\) is \(\epsilon\)-adjoint canonical.  Let
\[
\pi:(X,\cF)\to(Y,\cG)
\]
be the minimal resolution over $(Y,\cG,p)$, with exceptional divisor \(E\).  Then \(E\) is one
of the following:
\begin{enumerate}
\item an \(\cF\)-chain;
\item a chain of \((-2)\)-\(\cF\)-curves;
\item  the bad-tail chain, with all components of self-intersection \(-2\);
\item an \(\cF\)-dihedral fork, with all components of self-intersection \(-2\);
\item a non-trivial \(\mathcal F\)-star chain centered at a smooth rational
non-\(\mathcal F\)-invariant curve \(C\) satisfying
\(\operatorname{tang}(\mathcal F,C)=0\)
and \(C^2=-1\).
\end{enumerate}
\end{theorem}

The detailed form is stated and proved in
Theorem~\ref{thm:adjoint-canonical-classification}, after applying the
\(K_{\cF}+\epsilon K_X\)-effectivity filter to
Theorem~\ref{thm:adjoint-lc-classification}.  
The only configuration in the list
which need not be foliated canonical is the non-trivial \(\cF\)-star chain in item {\rm(5)};
see Definition~\ref{def:F-star-graph}.

As a consequence, we obtain the following stability statement.

\begin{proposition}\label{prop:intro-can-stability}
Fix \(\epsilon\in(0,1/4)\).  Suppose that \(p\) is an
\(\epsilon\)-adjoint canonical singularity of \((X,\cF)\).  Then \(p\) is log
canonical as a foliated singularity and klt as a surface singularity.
\end{proposition}

The proof is given in Corollary~\ref{cor:adjoint-canonical-stability}.
The numerical origin of the wall
\(1/4\) is explained in Remark~\ref{rem:canonical-wall-origin}; its sharpness is
realized in Example~\ref{ex:e=1/4}.

The same wall also controls the negative part in the Zariski decomposition of
the adjoint divisor \(K_{\mathcal F}+\epsilon K_X\).  At the adjoint terminal
level, the local classification gives an adjoint analogue of McQuillan's
description of the negative part of \(K_{\mathcal F}\): below the wall \(1/4\),
the classical \(\cF\)-chain picture is preserved, with exactly one genuinely
adjoint additional block.

\begin{proposition}[Adjoint negative part below \(1/4\)] \label{prop:intro-adjoint-negative-part} 
Let \((X,\cF)\) be a relatively minimal foliated surface in the sense of
Definition~\ref{def:relatively-minimal}.
Fix $\epsilon\in(0,\frac14)$.
Assume that \(K_{\cF}+\epsilon K_X\) is pseudo-effective, and write its Zariski decomposition as 
\[ 
K_{\cF}+\epsilon K_X=P_\epsilon+N_\epsilon.
\] 

Then every connected component of \(\operatorname{Supp}N_\epsilon\) is either an \(\cF\)-chain 
or a non-trivial \(\cF\)-star chain centered at a smooth rational non-\(\cF\)-invariant curve \(C\) satisfying 
\(\operatorname{tang}(\cF,C)=0\) and \(C^2=-1\). 
\end{proposition}

This is proved in Proposition~\ref{prop:adjoint-negative-part}. 
The \(\cF\)-star-chain case is genuinely adjoint and cannot be excluded
uniformly in \(\epsilon\); see Example~\ref{ex:epsilon-ad-can-NOT-can}.
The corresponding adjoint minimal model program and adjoint canonical model are
discussed in Corollaries~\ref{cor:adjoint-MMP} and
\ref{coro:adjoint-canonical-model}.

\subsection{Interpolated lc thresholds and the \(1/6\)-gap}

The wall \(1/5\) also admits an interpretation in terms of the interpolated lc
threshold introduced by McKernan \cite{McK22}; see also
\cite[Appendix~A]{CHL+24}.

For a foliated surface \((X,\cF)\), define
\[
t_0:=\sup\left\{\,t\in[0,1]\ \middle|\ 
tK_{\cF}+(1-t)K_X \text{ is log canonical}
\,\right\}.
\]
The divisor \(K_{\cF}+\epsilon K_X\) corresponds to the above expression through
the relation
\[
\epsilon=\frac{1-t}{t},\qquad t>0.
\]
Applying Proposition~\ref{prop:intro-lc-stability} through this change of
variables gives the following gap statement.

\begin{proposition}\label{prop:gap}
The interpolated lc threshold \(t_0\) satisfies
\[
t_0=1
\qquad\text{or}\qquad
t_0\le \frac56.
\]
\end{proposition}

More generally, the following \(1\)-gap conjecture for interpolated lc thresholds
was formulated in \cite[Conjecture~A.6]{CHL+24}.

\begin{conjecture}[\(1\)-gap for interpolated lc thresholds]
Let \(d\) be a positive integer.  Then there exists a positive real number
\(\tau=\tau(d)\) satisfying the following property.

Let \((X,\cF,t)\) be an adjoint foliated structure of dimension \(d\).  Assume
that \(tK_{\cF}+(1-t)K_X\) is log canonical for some \(t>1-\tau\) and that
\(K_{\cF}\) is \(\mathbb Q\)-Cartier.  Then \(\cF\) is log canonical.
\end{conjecture}

In dimension \(2\), our results show that the optimal value is
\[
\tau(2)=\frac16.
\]

\paragraph{\bf Organization of the paper.}
Section~\ref{sec:preliminaries} reviews background on foliations, adjoint singularities, and the
classifications of log canonical surface and foliated surface singularities.
Section~\ref{sec:numerical-reduction} develops the numerical projection and peeling tools for negative
definite configurations.  
Section~\ref{sec:adjoint-lc-reduction} runs the reduction algorithm to classify
\(\epsilon\)-adjoint log canonical singularities for
\(\epsilon\in(0,1/3)\) and identifies the first wall \(1/5\).
Section~\ref{sec:adjoint-canonical} imposes the adjoint canonical condition on this list, determines the
wall \(1/4\), and studies the corresponding adjoint negative part and adjoint MMP.
Section~\ref{sec:examples} presents examples illustrating the sharpness of both thresholds.

\section{Preliminaries}\label{sec:preliminaries}

\subsection{Foliations, singularities, and resolutions}

\subsubsection{Definition of foliations and singularities}

A \emph{foliation} \(\mathcal F\) on a normal surface \(X\) is a rank-one saturated
subsheaf \(T_{\mathcal F}\subset T_X\). Since \(T_{\mathcal F}\) is reflexive, it
determines a Weil divisor \(K_{\mathcal F}\) by
\(\mathcal O_X(-K_{\mathcal F})\simeq T_{\mathcal F}\).

A point $p\in X$ is called a \emph{singularity} of the foliation $\cF$ if either a singular point of $X$ or a point at which the quotient $T_X/T_{\cF}$ is not locally free. 
We say $p$ is a \emph{regular point} of $\cF$ if $p$ is a smooth point of $X$ and  $T_X/T_{\cF}$ is locally free at $p$.
\smallskip

Let $p$ be a singularity of $\cF$, at which $X$ is smooth. Locally, $\cF$ is defined by a vector field
  \begin{equation}
  \nu=A(x,y)\,\partial_x+B(x,y)\,\partial_y,
  \end{equation}
  where $p=(0,0)$. 
  The two eigenvalues $\lambda_1,\lambda_2$ of the linear part $(D\nu)(p)$ of $\nu$ at $p$ are well defined.
  \begin{definition}
  Let $p$ be a singularity of $\cF$ (at which $X$ is smooth) and let $\lambda_1,\lambda_2$ be as above.
  \begin{enumerate}
  \item The singularity $p$ is called \emph{non-degenerate} if the two eigenvalues $\lambda_1,\lambda_2$ are both nonzero.
  \item The singularity $p$ is called \emph{reduced} if one of the two eigenvalues, say $\lambda_2$, is nonzero and the quotient 
  $\lambda=\lambda_1/\lambda_2$ is not a positive rational number. 
  In particular, if $\lambda=0$, we call $p$ a \emph{saddle-node}.
  \end{enumerate}

  A foliation $\cF$ is said to be \emph{reduced} if any singularity of $\cF$ is reduced.
\end{definition}

By Seidenberg's theorem \cite{Sei68,Bru15}, after a sequence of blow-ups every
foliation on a surface has only reduced singularities.

\subsubsection{ Minimal resolutions.}
\begin{definition}
    Let $(X, \mathcal{F})$ be a foliated surface. A curve $C\subset X$ is said to be \emph{$\mathcal{F}$-exceptional} if the following conditions are satisfied:
    \begin{enumerate}
      \item $C$ is a smooth rational curve with self-intersection $-1$;
      \item the contraction of $C$ to a point $p$ yields a foliated surface $(X', \mathcal{F}')$, where $p$ is either a \emph{regular point} or a \emph{reduced singularity} of $\mathcal{F}'$.
    \end{enumerate}
\end{definition}

\begin{definition}\label{def:min-resolution}
Let $(X,\cF)$ be a foliated surface and let $p$ be a singularity of $(X,\cF)$.
A \emph{resolution of $(X,\cF,p)$} is a birational morphism
\[
\pi:(Y,\cG)\longrightarrow (X,\cF)
\]
such that \(Y-\pi^{-1}(p)\cong X-p\), and every point
\(q\in \pi^{-1}(p)\) is a smooth point of \(Y\) at which the induced foliation
\(\cG:=\pi^*\cF\) is either reduced or regular. We say that \(\pi\) is
\emph{minimal} if no \(\cG\)-exceptional curve is contained in \(\pi^{-1}(p)\).
\end{definition}

\subsection{Index formulas for invariant and non-invariant curves}

Let $(X,\cF)$ be a foliated surface with $X$ smooth.

A reduced curve $C\subseteq X$ is said to be \emph{$\cF$-invariant} if the inclusion
$T_{\cF}|_C \rightarrow T_X|_C$ factors through $T_C$, where $T_C$ is the tangent bundle of $C$.

\subsubsection{Non-invariant curves}
For a non-$\cF$-invariant curve  $C$ and $p\in C$, we can define the \emph{tangency order}  of $\cF$ along $C$ at $p$ to be
$${\rm tang} (\cF,C,p):=\dim_{\bC}\frac{\cO_p}{\langle f,\nu(f)\rangle},$$   
where $\nu$  is the local  generator of $\cF$ at $p$ and $f=0$ is the local equation of $C$.
We define  ${\rm tang}(\cF,C):=\sum_{p\in C}{\rm tang}(\cF,C,p)$. 
    
\begin{proposition}\label{prop:tang-index}
  We have
  \begin{equation}\label{equ:tang}
  {\rm tang}(\cF,C)=K_{\cF}\cdot C+C^2\, (\geq0).
  \end{equation} 
\end{proposition}    
\begin{proof}
See \cite[Proposition 2.2]{Bru15}.
\end{proof} 

\begin{remark}
${\rm tang}(\cF,C)=0$ implies that $C$ is smooth.
\end{remark}

\subsubsection{Invariant curves}
Let $C$ be an $\cF$-invariant curve on $X$, and let $p \in C$ be a point.
In a neighborhood of $p$, the foliation $\cF$ is defined by a 1-form $\omega$, and $C$ is locally given by the equation $f = 0$.
Since $C$ is $\cF$-invariant, we may write
$$g\omega=h df+f\eta,$$
where \(\eta\) is a holomorphic \(1\)-form and \(g,h\) are holomorphic functions
defined around \(p\), with \(h\) and \(f\) coprime.
We define 
\[
{\rm Z}(\cF,C,p):=\text{vanishing order of $\left.\frac{h}{g}\right|_C$ at $p$},\qquad
{\rm CS}(\cF,C,p):={\rm Res}_{p}\left\{\left.-\frac{\eta}{h}\right|_C\right\}.
\]
By definition, both ${\rm Z}(\cF,C,p)$ and ${\rm CS}(\cF,C,p)$ vanish if $p$ is not a singularity of $\cF$.
If $p$ is a reduced singularity of $\cF$, then ${\rm Z}(\cF,C,p)\geq0$.

Let
\({\rm Z}(\mathcal F,C):=\sum_{p\in C}Z(\mathcal F,C,p)\) and
\({\rm CS}(\mathcal F,C):=\sum_{p\in C}CS(\mathcal F,C,p)\).

\begin{proposition}\label{prop:Z-index-and-CS-index}
Let $C$ be an $\cF$-invariant curve on $X$. Then
\[
{\rm Z}(\cF,C)=K_{\cF}\cdot C+2-2p_a(C),\qquad
{\rm CS}(\cF,C)=C^2,
\]
where $p_a(C)$ denotes the arithmetic genus of $C$. The second equality above is called the \emph{Camacho-Sad formula}.
\end{proposition}
\begin{proof}
See \cite[Proposition 2.3 and Theorem 3.2]{Bru15}.
\end{proof}

\begin{remark}\label{rem:local-Z-values}
We shall use the standard local facts that, at a reduced non-degenerate
singularity, each invariant separatrix has
\({\rm Z}(\mathcal F,C,p)=1\), while at a saddle-node the strong separatrix has
\({\rm Z}(\mathcal F,C,p)=1\) and a weak separatrix, when it exists, has
\({\rm Z}(\mathcal F,C,p)\ge2\).  See \cite[pp.~30--31]{Bru15}.
\end{remark}

Next, we recall the separatrix theorem.
\begin{theorem}[Separatrix theorem]\label{thm:separatrix}
  Let $\cF$ be a foliation on a smooth projective surface $X$ and let $C\subset X$ be a connected compact $\cF$-invariant curve such that:
  \begin{itemize}
  \item[\rm(i)] All the singularities of $\cF$ on $C$ are reduced (in particular, $C$ has only normal crossing singularities);
  \item[\rm(ii)] If $C_1,\dots,C_n$ are the irreducible components of $C$, then the intersection matrix $(C_iC_j)_{1\leq i,j\leq n}$ is negative definite and the dual graph $\Gamma$ is a tree.
  \end{itemize}
  Then there exists at least one point $p\in C\cap{\rm Sing}(\cF)$ and a separatrix through $p$ not contained in $C$.  
  \end{theorem}
  \begin{proof}
  See \cite[Theorem 3.4]{Bru15}.
  \end{proof}

\subsection{Foliated and adjoint singularities}

\subsubsection{Foliated pairs and triples}

Let \(X\) be a normal surface and let \(\mathcal F\) be a foliation on \(X\).
For an \(\mathbb R\)-divisor \(D\) on \(X\), write
\[
D=D^{\rm inv}_{\mathcal F}+D^{\rm n\!-\!inv}_{\mathcal F}
\]
for its decomposition into the \(\mathcal F\)-invariant and non-\(\mathcal F\)-invariant
parts.  When there is no risk of confusion, we omit the subscript
\(\mathcal F\).

We use the standard terminology of pairs and foliated pairs; in particular, a
foliated pair \((\mathcal F,\Delta)\) means that \(\Delta\ge0\) and
\(K_{\mathcal F}+\Delta\) is \(\mathbb R\)-Cartier.  For the adjoint divisors
considered below, we shall use the following variant.

\begin{definition}
A \emph{foliated triple} $(X, \mathcal{F}, \Delta)$ consists of a foliation $\mathcal{F}$ on a normal surface $X$ and an effective $\mathbb{R}$-divisor $\Delta \ge 0$ such that both
$K_{\mathcal{F}} + \Delta^{\mathrm{n\!-\!inv}}$ and $K_X + \Delta$ are $\mathbb{R}$-Cartier.
\end{definition}

\subsubsection{Foliated log canonical singularities}
For a birational morphism \(\pi:X'\to X\), write
\[
K_{\mathcal F'}=\pi^*(K_{\mathcal F}+\Delta)+
\sum_E a(E,\mathcal F,\Delta)E.
\]
The numbers \(a(E,\mathcal F,\Delta)\) are the \emph{discrepancies} of
\((\mathcal F,\Delta)\).

\begin{definition}\label{def:lc-foliation}
Let $X$ be a normal surface and let $(\cF,\Delta)$ be a foliated pair on $X$.
We say that $(\cF,\Delta)$ is
\emph{terminal} (resp.\ \emph{canonical}, \emph{klt}, \emph{log canonical}) if
\[
a(E,\cF,\Delta) > 0
\quad (\text{resp.\ } \ge 0,\ > -\iota(E)\ \text{and } \lfloor \Delta \rfloor = 0,\ \ge -\iota(E))
\]
for any birational morphism $\pi \colon X' \to X$ and for any prime divisor $E$
on $X'$, where
\begin{equation}\label{equ:def-iota(E)}
\iota(E) :=
\begin{cases}
1, & \text{if $E$ is not $\cF'$-invariant},\\
0, & \text{if $E$ is $\cF'$-invariant}.
\end{cases}
\end{equation}
\end{definition}

\subsubsection{\((\epsilon,\delta)\)-adjoint singularities}
\begin{definition}\label{def:K-(X,F,D)-e}
Let $(X, \cF, \Delta)$ be a foliated triple on a normal surface $X$, and fix
$\epsilon \ge 0$.
The \emph{$\epsilon$-adjoint divisor} of $(X, \cF, \Delta)$ is defined by
\[
K_{(X, \cF, \Delta), \epsilon}
:= (K_{\cF} + \Delta^{\mathrm{n\!-\!inv}})
+ \epsilon (K_X + \Delta).
\]
\end{definition}

\begin{definition}\label{def:adj-lc-sing}
Let $(X,\cF,\Delta)$ be a foliated triple.
Fix $\epsilon > 0$ and $\delta \ge 0$.
We say that $(X,\cF,\Delta)$ is
\emph{$(\epsilon,\delta)$-adjoint log canonical}
(resp.\ \emph{$(\epsilon,\delta)$-adjoint klt})
if for any birational morphism $\pi \colon X' \to X$,
writing
\[
K_{(X',\cF',\Delta'),\epsilon}
= \pi^* K_{(X,\cF,\Delta),\epsilon}
+ \sum_i a_i E_i,
\]
where the $E_i$ are $\pi$-exceptional divisors and
$\Delta' := \pi_*^{-1}\Delta$ is the strict transform of $\Delta$,
we have, for all $i$,
\[
a_i \ge (\iota(E_i)+\epsilon)(-1+\delta)
\]
\[(\text{resp.}\quad
a_i > (\iota(E_i)+\epsilon)(-1+\delta)
\quad \text{and} \quad \lfloor \Delta \rfloor = 0).
\]

When $\delta = 1$, we call $(\epsilon,\delta)$-adjoint log canonical
\emph{$\epsilon$-adjoint canonical}.
When $\delta = 0$, we call $(\epsilon,\delta)$-adjoint log canonical simply
\emph{$\epsilon$-adjoint log canonical}.

Moreover, we say that $(\cF,\Delta)$ is \emph{$\epsilon$-adjoint terminal} if $a_i>0$ for all $i$.
\end{definition}

\begin{definition}\label{def:relatively-minimal}
Let \((X,\mathcal F)\) be a smooth projective foliated surface such that \(\mathcal F\) has
at most canonical singularities.  We say that \((X,\mathcal F)\) is
\emph{relatively minimal} if \(X\) contains no \(\mathcal F\)-invariant
\((-1)\)-curve \(E\) satisfying
\(K_{\mathcal F}\cdot E\le0\).
\end{definition}

\subsection{Basic exceptional configurations}
In this subsection, let \((X,\cF)\) be a foliated surface with \(X\) smooth.
The following configurations form the reference list for the classification
statements below.  They are the usual exceptional configurations appearing in
foliated log canonical surface singularities, together with the terminology needed
to describe the adjoint lists.

\begin{definition}[$\cF$-chain]\label{def:F-chain}
A divisor $\Theta=\Gamma_1+\cdots+\Gamma_r$ is called an \emph{$\cF$-chain} if
\begin{enumerate}
  \item $\Theta$ is a Hirzebruch--Jung string (cf.~\cite[Ch.III; Sec.5]{BPV04});
  \item each irreducible component $\Gamma_j$ is $\cF$-invariant;
  \item ${\rm Sing}(\cF)\cap \Theta$ consists of reduced, non-degenerate singularities;
  \item ${\rm Z}(\cF,\Gamma_1)=1$ and ${\rm Z}(\cF,\Gamma_i)=2$ for all $i\ge 2$.
\end{enumerate}
The curve $\Gamma_1$ is called the \emph{first curve} of the $\cF$-chain
(see Figure~\ref{fig:Fchain}).
\end{definition}

\begin{figure}[htbp]
  \centering
  \def\svgwidth{\columnwidth}
  \scalebox{0.5}{
\begingroup%
  \makeatletter%
  \providecommand\color[2][]{%
    \errmessage{(Inkscape) Color is used for the text in Inkscape, but the package 'color.sty' is not loaded}%
    \renewcommand\color[2][]{}%
  }%
  \providecommand\transparent[1]{%
    \errmessage{(Inkscape) Transparency is used (non-zero) for the text in Inkscape, but the package 'transparent.sty' is not loaded}%
    \renewcommand\transparent[1]{}%
  }%
  \providecommand\rotatebox[2]{#2}%
  \newcommand*\fsize{\dimexpr\f@size pt\relax}%
  \newcommand*\lineheight[1]{\fontsize{\fsize}{#1\fsize}\selectfont}%
  \ifx\svgwidth\undefined%
    \setlength{\unitlength}{126.14618764bp}%
    \ifx\svgscale\undefined%
      \relax%
    \else%
      \setlength{\unitlength}{\unitlength * \real{\svgscale}}%
    \fi%
  \else%
    \setlength{\unitlength}{\svgwidth}%
  \fi%
  \global\let\svgwidth\undefined%
  \global\let\svgscale\undefined%
  \makeatother%
  \begin{picture}(1,0.12199074)%
    \lineheight{1}%
    \setlength\tabcolsep{0pt}%
    \put(0.56121918,0.0263005){\color[rgb]{0,0,0}\makebox(0,0)[lt]{\lineheight{1.25}\smash{\begin{tabular}[t]{l}$\cdots\cdots$\end{tabular}}}}%
    \put(0,0){\includegraphics[width=\unitlength,page=1]{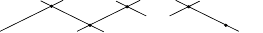}}%
    \put(0.25219959,0.07537052){\color[rgb]{0,0,0}\makebox(0,0)[lt]{\lineheight{1.25}\smash{\begin{tabular}[t]{l}$\Gamma_2$\end{tabular}}}}%
    \put(0.06765842,0.01258831){\color[rgb]{0,0,0}\makebox(0,0)[lt]{\lineheight{1.25}\smash{\begin{tabular}[t]{l}$\Gamma_1$\end{tabular}}}}%
    \put(0.78870066,0.0677606){\color[rgb]{0,0,0}\makebox(0,0)[lt]{\lineheight{1.25}\smash{\begin{tabular}[t]{l}$\Gamma_r$\end{tabular}}}}%
  \end{picture}%
\endgroup%
}
  \caption{$\cF$-chain with the first curve $\Gamma_1$}
  \label{fig:Fchain}
\end{figure}

\begin{definition}\label{def:-2Fcurve}
An irreducible curve $C$ is called a \emph{$(-1)$-$\cF$-curve}
(resp.\ a \emph{$(-2)$-$\cF$-curve}) if
\begin{enumerate}
  \item $C$ is a smooth rational $\cF$-invariant curve, and
  \item ${\rm Z}(\cF,C)=1$ (resp.\ ${\rm Z}(\cF,C)=2$).
\end{enumerate}
The terminology refers to the value of \({\rm Z}(\cF,C)\), not to the
self-intersection of \(C\).
\end{definition}

\begin{remark}\label{rem:minus-F-curves}
We shall use two standard facts.  First, a \((-1)\)-\(\mathcal F\)-curve with
self-intersection \(-1\) is \(\mathcal F\)-exceptional, and hence cannot occur on
a minimal resolution.  Secondly, a \((-2)\)-\(\mathcal F\)-curve with
self-intersection \(-1\) has one non-exceptional case: after contraction the
induced foliation is locally defined by
\[
\omega=(x+y)\,dx-x\,dy,
\]
equivalently the curve contains a unique saddle-node and is its weak separatrix.
See \cite[pp.~7--8]{Bru15}.
\end{remark}

\begin{definition}\label{def:bad-tail}
A curve $C$ is called a \emph{bad tail} if
\begin{enumerate}
  \item $C$ is a smooth rational irreducible $\cF$-invariant curve with
  ${\rm Z}(\cF,C)=3$ and $C^2\le -2$;
  \item $C$ intersects two $(-1)$-$\cF$-curves whose self-intersections are both $-2$.
\end{enumerate}
(See Figure~\ref{fig:badtail}.)
\end{definition}

\begin{figure}[htbp]
  \centering
  \def\svgwidth{\columnwidth}
  \scalebox{0.3}{
\begingroup%
  \makeatletter%
  \providecommand\color[2][]{%
    \errmessage{(Inkscape) Color is used for the text in Inkscape, but the package 'color.sty' is not loaded}%
    \renewcommand\color[2][]{}%
  }%
  \providecommand\transparent[1]{%
    \errmessage{(Inkscape) Transparency is used (non-zero) for the text in Inkscape, but the package 'transparent.sty' is not loaded}%
    \renewcommand\transparent[1]{}%
  }%
  \providecommand\rotatebox[2]{#2}%
  \newcommand*\fsize{\dimexpr\f@size pt\relax}%
  \newcommand*\lineheight[1]{\fontsize{\fsize}{#1\fsize}\selectfont}%
  \ifx\svgwidth\undefined%
    \setlength{\unitlength}{62.69560831bp}%
    \ifx\svgscale\undefined%
      \relax%
    \else%
      \setlength{\unitlength}{\unitlength * \real{\svgscale}}%
    \fi%
  \else%
    \setlength{\unitlength}{\svgwidth}%
  \fi%
  \global\let\svgwidth\undefined%
  \global\let\svgscale\undefined%
  \makeatother%
  \begin{picture}(1,0.17782633)%
    \lineheight{1}%
    \setlength\tabcolsep{0pt}%
    \put(0,0){\includegraphics[width=\unitlength,page=1]{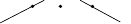}}%
    \put(0.41975037,0.0565801){\color[rgb]{0,0,0}\makebox(0,0)[lt]{\lineheight{1.25}\smash{\begin{tabular}[t]{l}$C$\end{tabular}}}}%
    \put(0.00633916,0.04892436){\color[rgb]{0,0,0}\makebox(0,0)[lt]{\lineheight{1.25}\smash{\begin{tabular}[t]{l}$-2$\end{tabular}}}}%
    \put(0.85230086,0.06040762){\color[rgb]{0,0,0}\makebox(0,0)[lt]{\lineheight{1.25}\smash{\begin{tabular}[t]{l}$-2$\end{tabular}}}}%
    \put(0,0){\includegraphics[width=\unitlength,page=2]{badtail.pdf}}%
  \end{picture}%
\endgroup%
}
  \caption{$C$ is a bad tail}
  \label{fig:badtail}
\end{figure}

\begin{definition}\label{def:F-dihedral-fork}
An \emph{$\cF$-dihedral fork} is a divisor
\[
\cD=\Gamma_1+\Gamma_2+\cdots+\Gamma_r,
\]
where
\begin{enumerate}
  \item $\Gamma_1$ and $\Gamma_2$ are $(-1)$-$\cF$-curves with self-intersection $-2$;
  \item $\Gamma_3$ is a bad tail, which is attached to a chain of $(-2)$-$\cF$-curves
  $\Gamma_4+\cdots+\Gamma_r$.
\end{enumerate}
(See Figure~\ref{fig:Dihedral}.)
\end{definition}

\begin{figure}[htbp]
  \centering
  \def\svgwidth{\columnwidth}
  \scalebox{0.4}{
\begingroup%
  \makeatletter%
  \providecommand\color[2][]{%
    \errmessage{(Inkscape) Color is used for the text in Inkscape, but the package 'color.sty' is not loaded}%
    \renewcommand\color[2][]{}%
  }%
  \providecommand\transparent[1]{%
    \errmessage{(Inkscape) Transparency is used (non-zero) for the text in Inkscape, but the package 'transparent.sty' is not loaded}%
    \renewcommand\transparent[1]{}%
  }%
  \providecommand\rotatebox[2]{#2}%
  \newcommand*\fsize{\dimexpr\f@size pt\relax}%
  \newcommand*\lineheight[1]{\fontsize{\fsize}{#1\fsize}\selectfont}%
  \ifx\svgwidth\undefined%
    \setlength{\unitlength}{145.34549677bp}%
    \ifx\svgscale\undefined%
      \relax%
    \else%
      \setlength{\unitlength}{\unitlength * \real{\svgscale}}%
    \fi%
  \else%
    \setlength{\unitlength}{\svgwidth}%
  \fi%
  \global\let\svgwidth\undefined%
  \global\let\svgscale\undefined%
  \makeatother%
  \begin{picture}(1,0.1190861)%
    \lineheight{1}%
    \setlength\tabcolsep{0pt}%
    \put(0.10743431,0.09554869){\color[rgb]{0,0,0}\makebox(0,0)[lt]{\lineheight{1.25}\smash{\begin{tabular}[t]{l}$\Gamma_1(-2)$\end{tabular}}}}%
    \put(0.61917962,0.02282636){\color[rgb]{0,0,0}\makebox(0,0)[lt]{\lineheight{1.25}\smash{\begin{tabular}[t]{l}$\cdots\cdots$\end{tabular}}}}%
    \put(0,0){\includegraphics[width=\unitlength,page=1]{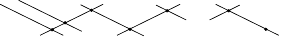}}%
    \put(-0.00072131,0.04560056){\color[rgb]{0,0,0}\makebox(0,0)[lt]{\lineheight{1.25}\smash{\begin{tabular}[t]{l}$\Gamma_2(-2)$\end{tabular}}}}%
    \put(0.35097975,0.0654145){\color[rgb]{0,0,0}\makebox(0,0)[lt]{\lineheight{1.25}\smash{\begin{tabular}[t]{l}$\Gamma_4$\end{tabular}}}}%
    \put(0.25521135,0.04229807){\color[rgb]{0,0,0}\makebox(0,0)[lt]{\lineheight{1.25}\smash{\begin{tabular}[t]{l}$\Gamma_3$\end{tabular}}}}%
    \put(0.8166121,0.05880981){\color[rgb]{0,0,0}\makebox(0,0)[lt]{\lineheight{1.25}\smash{\begin{tabular}[t]{l}$\Gamma_r$\end{tabular}}}}%
  \end{picture}%
\endgroup%
}
  \caption{$\cF$-dihedral fork}
  \label{fig:Dihedral}
\end{figure}

\begin{definition}[Elliptic Gorenstein leaf]\label{def:egl}
An \emph{elliptic Gorenstein leaf} is either a rational $\cF$-invariant curve
with a single node, or a cycle of $(-2)$-$\cF$-curves.
\end{definition}

\begin{definition}[$\cF$-star graph]\label{def:F-star-graph}
We say that
\[
\cT= C + \sum_{i=1}^s \Theta_i \qquad (s\geq0)
\]
is an \emph{$\cF$-star graph} (\emph{centered at $C$}) if
\begin{enumerate}
  \item $C$ is a smooth irreducible non-$\cF$-invariant curve with
  ${\rm tang}(\cF,C)=0$;
  \item for each $i=1,\dots,s$, $\Theta_i=\sum_{j=1}^{r(i)}\Gamma_{ij}$ is an $\cF$-chain
  whose first curve $\Gamma_{i1}$ satisfies
  \(C\cdot \Theta_i = C\cdot \Gamma_{i1} = 1\).
\end{enumerate}
We say that $\cT$ is of type $(s; m_1,\dots,m_s)$, where
\[
m_i = \det\!\bigl(-\Gamma_{ij}\cdot\Gamma_{ik}\bigr)_{1\le j,k\le r(i)}.
\]
The curve $C$ is called the \emph{center} of $\cT$
(see Figure~\ref{fig:F-star-graph}).

Moreover, if $s\leq2$, we call $\cT$ an \emph{$\cF$-star chain}; if $s\geq3$, we call $\cT$ an \emph{$\cF$-star fork}.
We say that an \(\cF\)-star chain is \emph{non-trivial} if it has at least two
irreducible components, equivalently if at least one \(\cF\)-chain is attached to
its center.
\end{definition}

\begin{figure}[htbp]
  \centering
  \def\svgwidth{\columnwidth}
  \scalebox{0.4}{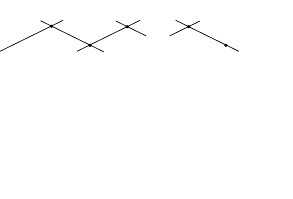}
  \caption{$\cF$-star graph centered at $C$}
  \label{fig:F-star-graph}
\end{figure}

\begin{figure}[htbp]
  \centering
  \def\svgwidth{\columnwidth}
  \scalebox{0.5}{
\begingroup%
  \makeatletter%
  \providecommand\color[2][]{%
    \errmessage{(Inkscape) Color is used for the text in Inkscape, but the package 'color.sty' is not loaded}%
    \renewcommand\color[2][]{}%
  }%
  \providecommand\transparent[1]{%
    \errmessage{(Inkscape) Transparency is used (non-zero) for the text in Inkscape, but the package 'transparent.sty' is not loaded}%
    \renewcommand\transparent[1]{}%
  }%
  \providecommand\rotatebox[2]{#2}%
  \newcommand*\fsize{\dimexpr\f@size pt\relax}%
  \newcommand*\lineheight[1]{\fontsize{\fsize}{#1\fsize}\selectfont}%
  \ifx\svgwidth\undefined%
    \setlength{\unitlength}{179.5825967bp}%
    \ifx\svgscale\undefined%
      \relax%
    \else%
      \setlength{\unitlength}{\unitlength * \real{\svgscale}}%
    \fi%
  \else%
    \setlength{\unitlength}{\svgwidth}%
  \fi%
  \global\let\svgwidth\undefined%
  \global\let\svgscale\undefined%
  \makeatother%
  \begin{picture}(1,0.08654667)%
    \lineheight{1}%
    \setlength\tabcolsep{0pt}%
    \put(0.47494003,0.02339617){\color[rgb]{0,0,0}\makebox(0,0)[lt]{\lineheight{1.25}\smash{\begin{tabular}[t]{l}$C$\end{tabular}}}}%
    \put(-0.00117276,0.0514598){\color[rgb]{0,0,0}\makebox(0,0)[lt]{\lineheight{1.25}\smash{\begin{tabular}[t]{l}$\Theta_1$\end{tabular}}}}%
    \put(0.87103167,0.02807308){\color[rgb]{0,0,0}\makebox(0,0)[lt]{\lineheight{1.25}\smash{\begin{tabular}[t]{l}$\Theta_2$\end{tabular}}}}%
    \put(0.66415483,0.033362){\color[rgb]{0,0,0}\makebox(0,0)[lt]{\lineheight{1.25}\smash{\begin{tabular}[t]{l}$\cdots\cdots$\end{tabular}}}}%
    \put(0,0){\includegraphics[width=\unitlength,page=1]{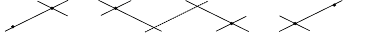}}%
    \put(0.1830175,0.01771634){\color[rgb]{0,0,0}\makebox(0,0)[lt]{\lineheight{1.25}\smash{\begin{tabular}[t]{l}$\cdots\cdots$\end{tabular}}}}%
  \end{picture}%
\endgroup%
}
  \caption{$\cF$-star chain centered at $C$, where $\Theta_1$ and $\Theta_2$ may be empty.}
  \label{fig:F-star-chain}
\end{figure}

\begin{remark}
The configurations introduced above arise naturally in the classification of
foliated surface singularities.  The terms \(\mathcal F\)-dihedral fork and
\(\mathcal F\)-star graph are used here as convenient names for the corresponding
configurations appearing below.
\end{remark}

\subsection{Classification of surface and foliated surface singularities}
We recall the two classical classification results used as reference lists for the
\(\epsilon\)-adjoint classification below.

\begin{theorem}\label{thm:lc-surface}
Let \(p\) be a log canonical singularity of a normal surface \(Y\). Let
\(\pi:X\to Y\) be the minimal resolution of \((Y,p)\), with exceptional divisor
\(E=\bigcup_i E_i\). Then \(E\) is one of the following:
\begin{enumerate}
\item a Hirzebruch--Jung chain;
\item an irreducible curve \(C\), where \(C\) is either a smooth elliptic curve or
a rational curve with one node;
\item a cycle of smooth rational curves;
\item a double-fork configuration of smooth rational curves, whose dual graph
consists of a chain joining two trivalent vertices, each carrying one additional
\((-2)\)-end component, and whose components all have self-intersection at most
\(-2\);
\item a single fork centered at a smooth rational curve \(C\), such that
\(E\setminus C\) consists of pairwise disjoint Hirzebruch--Jung chains
\(\Theta_1,\ldots,\Theta_l\) of determinant type \((m_1,\ldots,m_l)\), where
one of the following holds:
\begin{enumerate}
\item \(l=3\) and
$(m_1,m_2,m_3)=(2,2,n)$,\ $n\ge2$, $(2,3,3)$, $(2,3,4)$, or $(2,3,5)$;
\item \(l=3\) and
$(m_1,m_2,m_3)=(2,3,6)$, $(2,4,4)$, or $(3,3,3)$;

\item \(l=4\) and $(m_1,m_2,m_3,m_4)=(2,2,2,2)$.
\end{enumerate}
\end{enumerate}
The non-klt cases are precisely types {\rm (2)}, {\rm (3)}, {\rm (4)}, and
the second and third subcases in {\rm (5)}.
\end{theorem}
\begin{proof}
See \cite{Alex92} or \cite[Sec.~4.1]{KM98}.
\end{proof}

\begin{theorem}\label{thm:foliated-lc-classification}
Let \((Y,\cG,p)\) be a germ of a foliated surface with log canonical
foliated singularity at \(p\), and let
\(\pi:(X,\cF)\to(Y,\cG)\)
be the minimal resolution with exceptional divisor \(E=\cup_i E_i\).
Then \(E\) is one of the following:
\begin{enumerate}
\item an \(\cF\)-chain; see Definition~\ref{def:F-chain};

\item a chain of \((-2)\)-\(\cF\)-curves; see
Definition~\ref{def:-2Fcurve};

\item a chain of three \(\cF\)-invariant curves
\(E_1+E_2+E_3\),
where \(E_1,E_3\) are \((-1)\)-\(\cF\)-curves and \(E_2\) is a bad tail;
see Definition~\ref{def:bad-tail};

\item an \(\cF\)-dihedral fork; see Definition~\ref{def:F-dihedral-fork};

\item an elliptic Gorenstein leaf; see Definition~\ref{def:egl};

\item an \(\cF\)-star graph centered at a smooth non-\(\cF\)-invariant curve
\(C\) with
\(\operatorname{tang}(\cF,C)=0\); 
see Definition~\ref{def:F-star-graph}.
\end{enumerate}
Here, type {\rm(1)} is terminal and types {\rm(1)}--{\rm(5)} are canonical
as foliated singularities.
\end{theorem}
\begin{proof}
See \cite[Theorem~1.1]{YAChen23}. 
We note that the cases numbered {\rm(6)} and {\rm(7)} in
\cite[Theorem~1.1]{YAChen23} are combined here into the single case {\rm(6)}.
\end{proof}

\subsection{A negativity criterion on surfaces}

\begin{lemma}\label{lem:neg-def}
Let $D=\sum_{i=1}^n a_i C_i$ be a $\bQ$-divisor such that the intersection matrix
\(
(C_i \cdot C_j)_{1\le i,j\le n}
\)
is negative definite.
\begin{enumerate}
\item
If $D\cdot C_i \le 0$ for all $i=1,\dots,n$, then $D \ge 0$.
\item
If $E$ is an effective $\bQ$-divisor and $(E-D)\cdot C_j \le 0$ for all $j=1,\dots,n$,
then $E-D \ge 0$.
\end{enumerate}
\end{lemma}

\begin{proof}
Both statements follow from \cite[Lemma~14.9 and Lemma~14.15]{Luc01}.
\end{proof}

\section{Numerical reduction for negative definite configurations}
\label{sec:numerical-reduction}

This section develops a reduction procedure for numerical effectivity problems on negative definite configurations. 
The purpose is to isolate a numerical mechanism for classifying surface configurations satisfying a projection effectivity condition.

\medskip
\paragraph{\bf The numerical effectivity problem.}
Let \(D\) be an \(\mathbb R\)-divisor on a smooth surface \(X\), and let
\[
E=E_1+\cdots+E_n
\]
be a connected negative definite configuration.  We study the condition
\(M(D,E)\ge0\),
or equivalently, the condition that \(E\) is a \(D^{\ge0}\)-graph.  The aim is to
extract structural information on \(E\) from this numerical effectivity
condition.

\medskip

The basic operation is peeling.  The procedure is MMP-like in spirit: it is
driven by negative curves.  However, instead of contracting curves or changing
the surface, we keep the configuration \(E\) fixed and modify the divisor \(D\)
by removing known parts of its numerical projection.  After peeling off a
\(D^{>0}\)-graph \(G\), the residual divisor has zero intersection with every
component of \(G\).  Under the standing effectivity condition, the process stops
precisely when the current divisor has zero intersection with every component of
\(E\).

When \(D\) varies in a family, the same mechanism detects walls: a wall is a
parameter value at which one of the residual intersections vanishes for the first
time.  In the adjoint applications below, the sharp thresholds arise in this way.

The constructions in this section are purely numerical and do not use the foliation. 
They apply to any divisor \(D\) and any negative definite configuration \(E\) on a smooth surface. 
They should therefore be regarded as an intersection-theoretic tool for negative definite surface configurations, 
rather than as a construction specific to adjoint foliated singularities. 
They are closely related to the local Zariski decomposition on a resolution; see \cite[Theorem~A.1]{Sak84}.

Throughout this section, \(E\) is connected and negative definite, and we work under the standing assumption 
\[ 
M(D,E)\ge0. 
\]

\subsection{Numerical projection and peeling}
\label{subsec:numerical-projection-peeling}

Let
\[
G=C_1+\cdots+C_r\subset E
\]
be a subconfiguration.  Since the intersection matrix of \(G\) is negative
definite, for any divisor \(D\) on \(X\) there exists a unique divisor supported
on \(G\),
\[
M(D,G)=\sum_{i=1}^r a_iC_i,
\]
such that
\[
M(D,G)\cdot C_j=D\cdot C_j,\qquad j=1,\ldots,r.
\]

\begin{definition}[Numerical projection]\label{def:numerical-projection}
The divisor \(M(D,G)\) is called the \emph{numerical projection} of \(D\) onto
\(G\).  If all coefficients of \(M(D,G)\) are non-negative, we say that \(G\)
is a \emph{\(D^{\ge0}\)-graph}.  If all coefficients of \(M(D,G)\) are positive,
we say that \(G\) is a \emph{\(D^{>0}\)-graph}.  

If $M(D,G)=0$,
equivalently if \(D\cdot C=0\) for every component \(C\subset G\), we say that
\(G\) is a \emph{\(D^{=0}\)-graph}.
\end{definition}

If \(G=E\), then \(E\) is a \(D^{\ge0}\)-graph precisely when the divisor
supported on \(E\) whose intersection numbers agree with those of \(D\) is
effective.

\begin{definition}[Peeling]
\label{def:peeling}
Let \(G\subset E\) be a \(D^{>0}\)-graph.  We define the \emph{residual divisor}
after peeling off \(G\) by
\[
D^{(G)}:=D-M(D,G).
\]
We say that \(D^{(G)}\) is obtained from \(D\) by \emph{peeling off} \(G\).
\end{definition}

\begin{lemma}[Peeling identity]
\label{lem:peeling-identity}
Let \(G\subset E\) be a \(D^{>0}\)-graph.  Then
\[
D^{(G)}\cdot C=0
\]
for every irreducible component \(C\subset G\).  Moreover, 
for any component
\(C\subset E\), 
\[
D^{(G)}\cdot C=D\cdot C-M(D,G)\cdot C.
\]
\end{lemma}

\begin{proof}
This follows directly from the definition of \(M(D,G)\).
\end{proof}

\begin{lemma}[Peeling decomposition]
\label{lem:peeling-decomposition}
Let \(G\subset E\) be a \(D^{>0}\)-graph.  Then
\[
M(D,E)=M(D,G)+M(D^{(G)},E).
\]
In particular,
\[
M(D,E)>M(D^{(G)},E)\ge0.
\]
\end{lemma}

\begin{proof}
Since the intersection matrix of \(E\) is negative definite, it is
non-degenerate.  Comparing intersections with all components of \(E\) gives the
identity.

For the effectivity, Lemma~\ref{lem:neg-def}(2), applied to \(G\), gives
\[
M(D,G)\le M(D,E),
\]
because \(M(D,E)\geq0\) and for every component \(C\subset G\),
\[
\bigl(M(D,E)-M(D,G)\bigr)\cdot C
=
D^{(G)}\cdot C
=
0.
\]
Hence
\[
M(D^{(G)},E)=M(D,E)-M(D,G)\ge0.
\]
The strict inequality follows from \(M(D,G)>0\).
\end{proof}

\begin{remark}[What peeling does]
Peeling is a reduction of the divisor, not a modification of the configuration.
The configuration \(E\) is kept fixed, while \(D\) is replaced by
\[
D^{(G)}:=D-M(D,G).
\]
By Lemma~\ref{lem:peeling-decomposition}, the part of the numerical projection
supported on \(G\) has been removed, and the same effectivity problem continues
with the residual divisor \(D^{(G)}\) on the fixed configuration \(E\).

Iterating this operation, every peeled component becomes numerically trivial for
the new divisor.  In the applications, each peeled graph is generated by a
component with negative intersection.  Under the standing assumption
\(M(D_j,E)\ge0\), the absence of negative vertices forces \(M(D_j,E)=0\).
Thus the process stops precisely at the \((D_N)^{=0}\)-case.
\end{remark}

\subsection{Boundary contributions and residual clusters}
\label{subsec:boundary-residual}

A peeled subconfiguration remains visible to the unpeeled part through its
boundary contributions.  These boundary contributions determine the residual
intersections on the unpeeled components, and hence determine whether new
negative vertices remain.

Let \(G\subset E\) be a connected subconfiguration.  
When \(G\) is peeled off, the remaining part of \(E\) sees \(G\) only through
the components which meet it.
We define the \emph{boundary} of \(G\) in \(E\) by
\[
\partial_EG
:=
\{\, C\subset E-G \mid C\cdot G>0\,\}.
\]
Thus \(\partial_EG\) is the set of irreducible components of \(E\setminus G\)
which meet \(G\).

We define the \emph{external valency} of \(G\) by
\[
\rho(G):=(E-G)\cdot G=\sum_{C\in\partial_EG} C\cdot G.
\]
The number \(\rho(G)\) counts, with intersection multiplicities, the branches
of \(E\) meeting \(G\) from outside.

We say that \(G\) is \emph{closed} in \(E\) if \(\rho(G)=0\).  Since \(E\) is connected,
a connected subconfiguration \(G\subset E\) is closed if and only if \(G=E\).
For every boundary component \(C\in\partial_EG\), the peeling identity gives
\[
D^{(G)}\cdot C
=
D\cdot C-M(D,G)\cdot C.
\]
Thus \(M(D,G)\cdot C\) is precisely the boundary contribution of the peeled
subconfiguration \(G\) to the residual intersection of \(C\).
\medskip

We also need a \emph{local valency} notation for the unpeeled components.  For an
irreducible component \(C\subset E\), set
\[
d_E(C):=(E-C)\cdot C.
\]
When the ambient configuration is clear, we write simply \(d(C)\).  Thus
\(d(C)\) is the valency of \(C\) in \(E\), counted with intersection
multiplicities.

If \(D=D_0+\lambda E\), then
\[
D\cdot C=(D_0+\lambda C)\cdot C+\lambda d(C).
\]
In particular, by adjunction,
\[
(K_X+E)\cdot C=2p_a(C)-2+d(C).
\]
This elementary formula is the reason why valencies appear in the local
intersection tables below.
\medskip

Finally, suppose that one peels off a possibly disconnected \(D^{>0}\)-graph
\[
G=G_1+\cdots+G_s\subset E,
\]
where \(G_1,\ldots,G_s\) are its connected components.
It is useful to isolate
the part of \(G\) which actually contributes to a given residual component.

Let \(C\subset E\setminus G\) be an irreducible component.  The \emph{residual cluster}
centered at \(C\) is
\[
\mathcal R_C:=C+\sum_{G_i\cdot C>0}G_i.
\]
We set
\[
\rho_C:=\rho(\mathcal R_C).
\]
Thus \(\rho_C\) counts the branches of \(E\) still outside the cluster.  In
particular, \(\rho_C=0\) means that \(\mathcal R_C\) is closed; since \(E\) is
connected, this is equivalent to \(\mathcal R_C=E\).

If
\[
D_1:=D-M(D,G),
\]
then, since the components \(G_i\) are disjoint, we have
\[
M(D,G)=\sum_{i=1}^s M(D,G_i).
\]
Hence for every residual component \(C\subset E\setminus G\),
\[
D_1\cdot C
=
D\cdot C-\sum_{i=1}^s M(D,G_i)\cdot C.
\]
Only those components \(G_i\) contained in \(\mathcal R_C\) contribute to this
sum.  
Hence the residual cluster records exactly the local numerical data which
determine the residual intersection at \(C\).
\medskip

The preceding discussion applies to arbitrary peeled subconfigurations.  In the
applications, however, the subconfigurations to be peeled off are much more
specific: they are chains generated by a single component with negative
\(D\)-intersection. 

\subsection{Special chains and their estimates}
\label{subsec:special-chains}
We now introduce the elementary blocks generated by a single negative vertex.
Let
\[
S=C_1+\cdots+C_r\subset E
\]
be a chain.  We say that \(S\) is a \(D^{>0}\)-chain if it is a
\(D^{>0}\)-graph.

\begin{definition}[Special chain]
\label{def:special-chain}
A \(D^{>0}\)-chain
\(S=C_1+\cdots+C_r\)
is called \emph{special} if
\[
D\cdot C_1<0,\qquad D\cdot C_i\ge0 \quad (i\ge2).
\]
The curve \(C_1\) is called the first curve of \(S\).

A special \(D^{>0}\)-chain is \emph{maximal} if it is not properly contained in
a larger special \(D^{>0}\)-chain with the same first curve.
\end{definition}

\begin{lemma}[Extension criterion]
\label{lem:extension-criterion}
Let
\(
S=C_1+\cdots+C_r
\)
be a \(D^{>0}\)-chain, and let \(C\subset E\setminus S\) be a curve such that
\(S\cdot C=C_r\cdot C=1\).
Then \(S+C\) is a \(D^{>0}\)-chain if and only if
\[
D\cdot C<M(D,S)\cdot C.
\]
\end{lemma}

\begin{proof}
Write
\[
M(D,S)=\sum_{i=1}^r\gamma_iC_i,
\qquad
M(D,S+C)=\sum_{i=1}^{r+1}\gamma_i'C_i,
\]
where \(C_{r+1}:=C\).  Then
\[
\bigl(M(D,S+C)-M(D,S)\bigr)\cdot C_i
=
\begin{cases}
0, & i=1,\ldots,r,\\
D\cdot C-M(D,S)\cdot C, & i=r+1.
\end{cases}
\]
If \(D\cdot C<M(D,S)\cdot C\), then
\(
\bigl(M(D,S+C)-M(D,S)\bigr)\cdot C_i\le0
\)
for all \(i=1,\ldots,r+1\).  By Lemma~\ref{lem:neg-def}, this implies
\(M(D,S+C)>M(D,S)\).
Hence all coefficients of \(M(D,S+C)\) are positive, and \(S+C\) is a
\(D^{>0}\)-chain.

Conversely, if \(D\cdot C\ge M(D,S)\cdot C\), then the coefficient of \(C\) in
\(M(D,S+C)\) is non-positive.  Thus \(S+C\) is not a \(D^{>0}\)-chain.
\end{proof}

\begin{remark}
An analogous statement holds for $D^{\ge0}$-chains with the inequality $D\cdot C\le M(D,S)\cdot C$.
\end{remark}

The extension criterion tells us when a special chain can be prolonged.

To control the boundary contribution produced by a chain
\(S=C_1+\cdots+C_r\), we use the following standard determinant notation:
\[
\begin{cases}
\mu_0=\lambda_{r+1}:=1,\quad \mu_1=\lambda_r:=1,\\
\mu_k:=\det(-C_i\cdot C_j)_{1\le i,j\le k-1},
\quad k=2,\ldots,r+1,\\
\lambda_l:=\det(-C_i\cdot C_j)_{l+1\le i,j\le r},
\quad l=0,\ldots,r-1,\\
n:=\lambda_0=\mu_{r+1}.
\end{cases}
\]

\begin{proposition}[Control along a special chain]
\label{prop:special-chain-estimate}
Let
\(
S=C_1+\cdots+C_r
\)
be a special \(D^{>0}\)-chain.  Write
\[
M(D,S)=\sum_{i=1}^r\gamma_iC_i.
\]
Then:
\begin{enumerate}
\item For \(k=1,\ldots,r\),
\[
0<\gamma_k\le \frac{(-D\cdot C_1)\lambda_k}{n}.
\]

\item Let \(S_k=C_1+\cdots+C_k\).  For \(k\ge2\),
\[
D\cdot C_k<M(D,S_{k-1})\cdot C_k
\le \frac{-D\cdot C_1}{\mu_k}.
\]
\item Each \(S_k\) is a special \(D^{>0}\)-chain, and
\[
0<M(D,S_1)<M(D,S_2)<\cdots<M(D,S_r)=M(D,S).
\]
\end{enumerate}
\end{proposition}

\begin{proof}
Let \(M_1(S)\) be the unique divisor supported on \(S\) such that
\[
M_1(S)\cdot C_1=-1,\qquad M_1(S)\cdot C_i=0\quad (i\ge2).
\]
A direct computation gives
\[
M_1(S)=\sum_{k=1}^r\frac{\lambda_k}{n}C_k.
\]
Since
\[
\bigl((-D\cdot C_1)M_1(S)-M(D,S)\bigr)\cdot C_i\le0
\]
for all \(i\), the negativity criterion gives
\[
M(D,S)\le (-D\cdot C_1)M_1(S).
\]
This proves the coefficient estimate.

The remaining assertions follow by applying the same argument to the initial
subchains \(S_k\), together with the extension criterion.
\end{proof}

\begin{corollary}[Boundary contribution]
\label{cor:boundary-contribution}
With the notation of Proposition~\ref{prop:special-chain-estimate}, let
\(C\subset E\setminus S\) be a curve meeting \(S\) transversely only at \(C_r\).
Then
\[
M(D,S)\cdot C
=
\gamma_r
\le
\frac{-D\cdot C_1}{n}.
\]
If moreover
$D\cdot C_i=0$ for all $i\ge2$,
then equality holds.
\end{corollary}

\begin{remark}
If a curve \(C\subset E\setminus S\) meets \(S\) in several components, then
\[
M(D,S)\cdot C=\sum_{i=1}^r \gamma_i(C_i\cdot C).
\]
In the applications below, the special chains meet the residual component only
at their last component, and the preceding form is sufficient.
\end{remark}

In the adjoint application, the local intersection formulas will first identify
the possible negative vertices.  Each such vertex generates a maximal special
chain.  Proposition~\ref{prop:special-chain-estimate} and
Corollary~\ref{cor:boundary-contribution} control the boundary contributions
of these chains to the residual components.  
This is the numerical input used below to control the residual intersections
after peeling and to locate the adjoint walls.

\section{Adjoint log canonical singularities and the wall \(1/5\)}
\label{sec:adjoint-lc-reduction}

In this section we apply the numerical reduction developed in
Section~\ref{sec:numerical-reduction} to adjoint log canonical foliated surface
singularities.

Let \((Y,\mathcal G,p)\) be a germ of a foliated surface singularity, and fix
\[
0<\epsilon<\frac13.
\]
Let
\[
\pi:(X,\mathcal F,E=\sum_iE_i)\to (Y,\mathcal G,p)
\]
be the minimal resolution (cf.~Definition~\ref{def:min-resolution}), where
\(E=\sum_iE_i
\)
is the reduced exceptional divisor.  We set
\[
K_{\epsilon}:=K_{(X,\mathcal F,E),\epsilon}
=
(K_{\mathcal F}+E^{\mathrm{n\mbox{-}inv}})
+\epsilon(K_X+E).
\]

By Definition~\ref{def:adj-lc-sing}, the germ \((Y,\mathcal G,p)\) is
\(\epsilon\)-adjoint log canonical if and only if the divisor supported on \(E\)
whose intersections with the components of \(E\) agree with those of \(K_{\epsilon}\) is
effective.  Equivalently,
\[
(Y,\mathcal G,p)\text{ is }\epsilon\text{-adjoint log canonical}
\quad\Longleftrightarrow\quad
E\text{ is a }(K_{\epsilon})^{\ge0}\text{-graph}.
\]

Thus the problem is an instance of the numerical effectivity problem of Section~\ref{sec:numerical-reduction}. 
We solve it by running the negative-vertex reduction for $K_{\epsilon}$:
\begin{enumerate}
\item compute the local possibilities for components with \(K_{\epsilon}\cdot C\le0\);
\item extract the special \((K_{\epsilon})^{>0}\)-chains generated by negative vertices;
\item peel off the maximal special chains and obtain a residual divisor \(K_{\epsilon,1}\);
\item analyze the residual components satisfying \(K_{\epsilon,1}\cdot C\le0\).
\end{enumerate}
Steps (2) and (3) use the extension criterion and the boundary contribution
estimate from Lemma~\ref{lem:extension-criterion},
Proposition~\ref{prop:special-chain-estimate}, and
Corollary~\ref{cor:boundary-contribution}.

For a fixed \(\epsilon\), the reduction tracks the 
\(K_\epsilon\)-negative vertices.  
After the forced special chains have been peeled off, the residual intersections
are the remaining quantities to be controlled.  As \(\epsilon\) varies, one of
these residual intersections may become zero for the first time.  This first
critical value is the wall: a previously excluded pattern becomes admissible,
and a new configuration enters the list.

\subsection{Local numerical input: negative and zero vertices}
\label{subsec:local-input}

For a component \(C\subset E\), write
\[
d(C):=(E-C)\cdot C.
\]
Following the notation of $\iota(C)$ as in \eqref{equ:def-iota(E)}, we also set
\[
d_1(C):=(E^{\mathrm{n\mbox{-}inv}}-\iota(C)C)\cdot C,
\qquad
d_2(C):=(E^{\mathrm{inv}}-(1-\iota(C))C)\cdot C,
\]
so that
\[
d(C)=d_1(C)+d_2(C).
\]
Here \(d_1(C)\) and \(d_2(C)\) measure respectively the non-invariant and
invariant branches meeting \(C\).

\begin{lemma}[Non-invariant vertices]
\label{lem:non-inv-local}
Let \(C\subset E\) be a non-\(\mathcal F\)-invariant component. Then
\[
K_{\epsilon}\cdot C
=
\epsilon(2p_a(C)-2+d(C))
+\operatorname{tang}(\mathcal F,C)+d_1(C).
\]
If \(K_{\epsilon}\cdot C\le0\), then
\(\operatorname{tang}(\mathcal F,C)=d_1(C)=0\).
Moreover, one of the following holds:
\begin{enumerate}
\item \(C\) is a smooth elliptic curve. In this case
\(d(C)=K_{\epsilon}\cdot C=0\).
\item \(C\) is a smooth rational curve. In this case
\[
d(C)=d_2(C)\le2,
\qquad
K_{\epsilon}\cdot C=\epsilon(d(C)-2).
\]
Equivalently,
\[
\begin{array}{c|c}
d(C)=d_2(C) & K_{\epsilon}\cdot C\\
\hline
0 & -2\epsilon\\
1 & -\epsilon\\
2 & 0.
\end{array}
\]
\end{enumerate}
\end{lemma}

\begin{proof}
We compute
\begin{align*}
K_{\epsilon}\cdot C
&=(K_{\cF}+E^{\rm n\!-\!inv})\cdot C+\epsilon (K_X+E)\cdot C \\
&=(K_{\cF}+C)\cdot C+\epsilon (K_X+C)\cdot C
 +(E^{\rm n\!-\!inv}-C)\cdot C+\epsilon (E-C)\cdot C .
\end{align*}
By Proposition~\ref{prop:tang-index}, adjunction, and the definitions of
\(d_1(C)\) and \(d(C)\), this gives the formula.

The classification follows immediately from \(K_{\epsilon}\cdot C\le0\), the inequalities
\[
\operatorname{tang}(\mathcal F,C)\ge0,\qquad d_1(C)\ge0,
\]
and the assumption \(0<\epsilon<1/3\).  Indeed, either \(p_a(C)=1\), in which
case all remaining terms vanish and \(d(C)=0\), or \(p_a(C)=0\), in which case
\(\operatorname{tang}(\mathcal F,C)=d_1(C)=0\) and \(d(C)\le2\).  This gives the
two cases above.
\end{proof}

\begin{lemma}[Invariant vertices]
\label{lem:inv-local}
Let \(C\subset E\) be an \(\mathcal F\)-invariant component. Then
\[
K_{\epsilon}\cdot C
=
(1+\epsilon)(2p_a(C)-2+d(C))
+{\rm Z}(\mathcal F,C)-d_2(C).
\]
If \(K_{\epsilon}\cdot C\le0\), then one of the following holds.

\begin{enumerate}
\item \(C\) is an irreducible elliptic Gorenstein leaf. In this case
\(d(C)=K_{\epsilon}\cdot C=0\).

\item \(C\) is a smooth rational curve. In this case the possible numerical
data are exactly:
\[
\begin{array}{c|c|c|c|c}
d(C) & d_1(C) & d_2(C) & {\rm Z}(\mathcal F,C) & K_{\epsilon}\cdot C\\
\hline
0 & 0 & 0 & 1 & -(1+2\epsilon)\\
  & 0 & 0 & 2 & -2\epsilon\\
\hline
1 & 1 & 0 & 1 & -\epsilon\\
  & 0 & 1 & 1 & -(1+\epsilon)\\
  & 0 & 1 & 2 & -\epsilon\\
\hline
2 & 1 & 1 & 1 & 0\\
  & 0 & 2 & 2 & 0.
\end{array}
\]
Here \({\rm Z}(\mathcal F,C)=1,2\) correspond respectively to
\((-1)\)- and \((-2)\)-\(\mathcal F\)-curves.
\end{enumerate}
\end{lemma}

\begin{proof}
We compute
\begin{align*}
K_{\epsilon}\cdot C
&=(K_{\mathcal F}+E^{\rm n\!-\!inv})\cdot C+\epsilon(K_X+E)\cdot C\\
&=K_{\mathcal F}\cdot C+\epsilon(K_X+C)\cdot C
  +E^{\rm n\!-\!inv}\cdot C+\epsilon(E-C)\cdot C\\
&=K_{\mathcal F}\cdot C+\epsilon(2p_a(C)-2)+d_1(C)+\epsilon d(C)\\
&=K_{\mathcal F}\cdot C+\epsilon(2p_a(C)-2)+(1+\epsilon)d(C)-d_2(C).
\end{align*}
By Proposition~\ref{prop:Z-index-and-CS-index},
\[
K_{\mathcal F}\cdot C=2p_a(C)-2+{\rm Z}(\mathcal F,C),
\]
and the displayed formula follows.

Since each invariant branch meeting \(C\) contributes at least \(1\) to
\({\rm Z}(\mathcal F,C)\) (cf.~Remark~\ref{rem:local-Z-values}), we have
\({\rm Z}(\mathcal F,C)\ge d_2(C)\).  Together with \(0<\epsilon<1/3\), the
inequality \(K_{\epsilon}\cdot C\le0\) forces either \(p_a(C)=1\) and all
remaining terms vanish, or \(p_a(C)=0\) and \(d(C)\le2\).
The first case gives an irreducible elliptic
Gorenstein leaf.  In the smooth rational case, checking \(d(C)=0,1,2\) gives exactly
the table.
\end{proof}

\begin{corollary}[Gap principle]
\label{cor:gap}
Let \(C\subset E\) be a component with \(d(C)\ge1\).
\begin{enumerate}
\item If \(C\) is non-\(\mathcal F\)-invariant and
\(0\le K_{\epsilon}\cdot C<\epsilon\),
then \(K_{\epsilon}\cdot C=0\).
\item If \(C\) is \(\mathcal F\)-invariant and
\(0\le K_{\epsilon}\cdot C<\frac{1+\epsilon}{2}\),
then \(K_{\epsilon}\cdot C=0\).
\end{enumerate}
\end{corollary}

\begin{proof}
This follows directly from the two local formulas above and the assumption $0<\epsilon<1/3$.
\end{proof}

The local input has two immediate consequences for the reduction.  First,
negative vertices are the only possible starting points of the peeling process,
and every non-isolated negative vertex is an end vertex.  Second, the estimates
of Section~\ref{sec:numerical-reduction} control the boundary contribution of
any special chain by its first negative vertex.  In the range
\(0<\epsilon<1/3\), the gap principle further forces the later components of the
forced chains to be \(K_\epsilon\)-trivial, leading to the finite list below.

\subsection{Forced special chains}
\label{subsec:terminal-blocks}

Before extracting forced chains, we first remove the two base cases of the
reduction: the irreducible case and the case where the configuration is already
\(K_\epsilon\)-trivial.

\begin{lemma}[Base cases of the reduction]
\label{lem:base-cases}
Let \(E\) be a \((K_\epsilon)^{\ge0}\)-graph.
\begin{enumerate}
\item If \(E\) is irreducible, then \(E\) is one of the \(d=0\) cases in
Lemmas~\ref{lem:non-inv-local} and~\ref{lem:inv-local}.
\item If \(E\) is a \((K_\epsilon)^{=0}\)-graph and has more than one component,
then \(E\) is a cycle of \((-2)\)-\(\mathcal F\)-curves.
\end{enumerate}
\end{lemma}

\begin{proof}
The first statement is immediate from the local lists.  For the second, if
\(K_{\epsilon}\cdot C=0\) for every component \(C\subset E\), then the local lists force
every component to be a smooth rational curve with \(d(C)=2\).  Hence the dual
graph is a cycle.  The separatrix theorem, Theorem~\ref{thm:separatrix}, excludes the non-invariant
curve and the \((-1)\)-\(\mathcal F\)-curve in such a cycle, so all
components are \((-2)\)-\(\mathcal F\)-curves.
\end{proof}

We may therefore assume that \(E\) is reducible and not a
\((K_\epsilon)^{=0}\)-graph.  Since \(E\) is a \((K_\epsilon)^{\ge0}\)-graph, the
negativity criterion forces the existence of a component with
\(K_\epsilon\cdot C<0\).  By Lemmas~\ref{lem:non-inv-local} and
\ref{lem:inv-local}, such a component is an end vertex.  It therefore generates
a unique direction along which negativity can propagate.

\begin{proposition}[Forced special chains]
\label{prop:elementary-special-chains}
Let
\(
S=C_1+\cdots+C_r
\)
be a  special \((K_{\epsilon})^{>0}\)-chain (cf.~Definition~\ref{def:special-chain}) in \(E\), with first curve \(C_1\).
Then \(S\) is one of the following types.

\begin{enumerate}
\item[(A)] \(C_1\) is a \((-1)\)-\(\mathcal F\)-curve, and every \(C_i\) for
\(i\ge2\) is a \((-2)\)-\(\mathcal F\)-curve.  In this case
\(K_{\epsilon}\cdot C_1=-(1+\epsilon)\).

\item[(B)] Every \(C_i\) is a \((-2)\)-\(\mathcal F\)-curve.  In this case
\(K_{\epsilon}\cdot C_1=-\epsilon\).

\item[(C)] Every \(C_i\) for \(i<r\) is a \((-2)\)-\(\mathcal F\)-curve, and
\(C_r\) is a \((-1)\)-\(\mathcal F\)-curve.  In this case
\(K_{\epsilon}\cdot C_1=-\epsilon\).

\item[(D)] There is a unique non-\(\mathcal F\)-invariant component \(C_k\) in \(S\) with
\[
\operatorname{tang}(\mathcal F,C_k)=0,\qquad p_a(C_k)=0.
\]
The adjacent components \(C_{k-1}\) and \(C_{k+1}\), when they exist, are
\((-1)\)-\(\mathcal F\)-curves, and all remaining components are
\((-2)\)-\(\mathcal F\)-curves.  In this case
\(K_{\epsilon}\cdot C_1=-\epsilon\).
\end{enumerate}

Moreover, \(K_{\epsilon}\cdot C_i=0\) for every \(i\ge2\).  The unused valency
of \(S\) is at the last component \(C_r\).  More precisely,
\(\partial_E S=\{C\}\)
for a unique component \(C\subset E-S\), and
\(C\cdot S=C\cdot C_r=1\).

Furthermore, if \(S\) is of type {\rm (A)}, {\rm (B)}, or {\rm (D)}, then \(C\)
is \(\mathcal F\)-invariant; if \(S\) is of type {\rm (C)}, then \(C\) is
non-\(\mathcal F\)-invariant.
\end{proposition}

\begin{proof}
The first curve \(C_1\) satisfies \(K_{\epsilon}\cdot C_1<0\).  Since the
irreducible case has already been removed, \(C_1\) is not isolated.  By
Lemmas~\ref{lem:non-inv-local} and~\ref{lem:inv-local}, every non-isolated
negative vertex is an end vertex; hence
\(d(C_1)=1\).
Thus the local lists give only two possible values:
\[
K_{\epsilon}\cdot C_1=-(1+\epsilon)
\qquad\text{or}\qquad
K_{\epsilon}\cdot C_1=-\epsilon.
\]

We first assume that \(K_{\epsilon}\cdot C_1=-\epsilon\).  By the local lists,
\(C_1\) is one of the following three types: a \((-1)\)-\(\mathcal F\)-curve
with \((d,d_1,d_2)=(1,1,0)\), a \((-2)\)-\(\mathcal F\)-curve with
\((d,d_1,d_2)=(1,0,1)\), or a smooth rational non-\(\mathcal F\)-invariant curve
with \(\operatorname{tang}(\mathcal F,C_1)=0\) and
\((d,d_1,d_2)=(1,0,1)\).

For \(k\ge2\), put
\(S_{k-1}=C_1+\cdots+C_{k-1}\).  
Since \(S\) is a special
\((K_{\epsilon})^{>0}\)-chain, each initial subchain \(S_k\) is also a
special \((K_{\epsilon})^{>0}\)-chain.
Lemma~\ref{lem:extension-criterion}
and Proposition~\ref{prop:special-chain-estimate} give
\[
0\le K_{\epsilon}\cdot C_k
<
M(K_{\epsilon},S_{k-1})\cdot C_k
\le \frac{\epsilon}{\mu_k}
\le \epsilon .
\]
Hence \(K_{\epsilon}\cdot C_k=0\) for every \(k\ge2\), by
Corollary~\ref{cor:gap}.  The \(K_{\epsilon}\)-trivial vertices with positive
valency are read from Lemmas~\ref{lem:non-inv-local} and~\ref{lem:inv-local}.
They are:
\[
\begin{array}{c|c|c|c}
\text{component} & d & d_1 & d_2\\
\hline
\text{non-}\mathcal F\text{-invariant rational curve with }
\operatorname{tang}(\mathcal F,C)=0
& 2 & 0 & 2\\
(-1)\text{-}\mathcal F\text{-curve} & 2 & 1 & 1\\
(-2)\text{-}\mathcal F\text{-curve} & 2 & 0 & 2 .
\end{array}
\]
Thus every later component has valency \(2\), and its unused branch is determined
by \((d_1,d_2)\).

We now read the chain from left to right.  Starting from the chosen first curve
\(C_1\), the unused branch prescribed by the triple \((d,d_1,d_2)\) determines
the possible type of the next component.  If no non-\(\mathcal F\)-invariant
component occurs in \(S\), the chain remains in the invariant part.  Reading the
unused branches from the local list gives a chain of
\((-2)\)-\(\mathcal F\)-curves, possibly ending with one
\((-1)\)-\(\mathcal F\)-curve.  This gives types {\rm (B)} and {\rm (C)}.
If a non-\(\mathcal F\)-invariant component occurs, then all branches adjacent to
it are \(\mathcal F\)-invariant.  The invariant components adjacent to it are
\((-1)\)-\(\mathcal F\)-curves, while all remaining invariant components are
\((-2)\)-\(\mathcal F\)-curves.  This gives type {\rm (D)}.  The remaining formal
continuations are excluded by Theorem~\ref{thm:separatrix}.

It remains to consider the case
\(K_{\epsilon}\cdot C_1=-(1+\epsilon)\).
Then \(C_1\) is a \((-1)\)-\(\mathcal F\)-curve with
\((d,d_1,d_2)=(1,0,1)\), so its
unique adjacent branch is \(\mathcal F\)-invariant.  Since \(C_1\) is not
\(\mathcal F\)-exceptional on the minimal resolution,
\(\mu_2=-C_1^2\ge2\).
Therefore
\[
0\le K_{\epsilon}\cdot C_2
<
M(K_{\epsilon},C_1)\cdot C_2
\le \frac{1+\epsilon}{\mu_2}
\le \frac{1+\epsilon}{2}.
\]
By the invariant gap in Corollary~\ref{cor:gap}, we get
\(K_{\epsilon}\cdot C_2=0\). 
The local lists then leave only the \((-1)\)- and \((-2)\)-\(\mathcal F\)-curve
alternatives; the \((-1)\)-\(\mathcal F\)-curve alternative is excluded by
Theorem~\ref{thm:separatrix}.  Hence \(C_2\) is a
\((-2)\)-\(\mathcal F\)-curve.
Since \(C_2\) is not
\(\mathcal F\)-exceptional, \(C_2^2\le -2\).  Repeating the same argument
inductively gives that every \(C_k\) for \(k\ge2\) is a
\((-2)\)-\(\mathcal F\)-curve.  Hence \(S\) is of type {\rm (A)}.

Finally, the local lists give
\(d(C_1)=1\) and \(d(C_k)=2\) for \((k\ge2)\).
Thus all valencies of \(S\) are used inside the chain except for one at the last
component \(C_r\).  Hence \(S\) has a unique boundary component
\(C\in\partial_E S\), and
\[
C\cdot S=C\cdot C_r=1.
\]
The pair \((d_1(C_r),d_2(C_r))\) determines the type of this unused branch:
types {\rm (A)}, {\rm (B)}, and {\rm (D)} attach to an
\(\mathcal F\)-invariant residual component, while type {\rm (C)} attaches to a
non-\(\mathcal F\)-invariant residual component.
\end{proof}

\begin{remark}[Determinants of forced special chains]
\label{rem:determinant-one-forced-chain}
For a forced special chain \(S=C_1+\cdots+C_r\), its determinant means
\[
\det(-C_i\cdot C_j)_{1\le i,j\le r}.
\]
Forced special chains of types {\rm (B)} and {\rm (D)} may have determinant
\(1\).  This is not an empty-chain convention.  Rather, it means that such a
chain has determinant \(1\), hence can be contracted to a smooth point on the
underlying surface.  However, the induced foliation at the contraction point
need not be reduced.  Hence such a chain is not excluded by the minimality
condition in Definition~\ref{def:min-resolution}.  This is precisely the
determinant-one phenomenon for foliated blow-downs described in
\cite[pp.~7--8]{Bru15}.

In the type {\rm (B)} case, this phenomenon has the following consequence which
will be used below.  If a type {\rm (B)} chain has determinant \(1\), then after
contracting it the induced foliation has a singularity defined locally by \(\nu=x\partial_x+(ny+x^n)\partial_y\); see \cite[pp.~8]{Bru15}. 
The only
separatrix not contained in the chain is forced to pass through the component
meeting the last curve of the chain.  Thus, when such a chain is attached to a
residual component \(C\), its external separatrix is forced to be \(C\).

By contrast, forced special chains of types {\rm (A)} and {\rm (C)} have
determinant at least \(2\).  We also note that types {\rm (A)} and {\rm (C)}
may have the same underlying \(\mathcal F\)-chain, although they are different as
special \((K_{\epsilon})^{>0}\)-chains and give different boundary contributions
in Corollary~\ref{cor:elementary-chain-contribution}.
\end{remark}

\begin{corollary}[Boundary contribution of a forced chain]
\label{cor:elementary-chain-contribution}
Let \(S=\sum_{i=1}^rC_i\) be a special \((K_{\epsilon})^{>0}\)-chain, and let
\(C\subset E-S\) be an irreducible component with \(C\cdot S>0\).  Then
\(C\cdot S=C\cdot C_r=1\) and
\[
M(K_{\epsilon},S)\cdot C=
\begin{cases}
\dfrac{1+\epsilon}{n}, & \text{if }S\text{ is of type }(A),\\[0.8em]
\dfrac{\epsilon}{n}, & \text{if }S\text{ is of type }(B),(C),(D),
\end{cases}
\]
where \(n=\det(-C_i\cdot C_j)_{1\leq i,j\leq r}\).
\end{corollary}

\begin{proof}
By Proposition~\ref{prop:elementary-special-chains}, all components after the
first one are \(K_{\epsilon}\)-trivial.  The statement follows from
Corollary~\ref{cor:boundary-contribution}.
\end{proof}

The reduction itself only uses the existence of special
\((K_\epsilon)^{>0}\)-chains and the estimate of their boundary contributions by
the first negative vertex.  The gap principle is used here to make the resulting
chains explicit in the range \(0<\epsilon<1/3\).

\subsection{Interaction and simultaneous peeling}
\label{subsec:interaction-peeling}

There are two ways in which the reduction can close.  The first is immediate:
two forced chains may interact, in which case the negative vertices already
close up into one of the known foliated log canonical types.  Such
configurations require no residual wall analysis.

\begin{lemma}[Interaction of forced chains]
\label{lem:interaction}
Suppose that two special \((K_{\epsilon})^{>0}\)-chains with distinct first curves
intersect or share a component.  Then \(E\) is one of the following:
\begin{enumerate}
\item an \(\mathcal F\)-chain;
\item a chain of \((-2)\)-\(\mathcal F\)-curves;
\item an \(\mathcal F\)-star chain centered at a smooth rational
non-\(\mathcal F\)-invariant curve \(C\) with
\(\operatorname{tang}(\mathcal F,C)=0\).
\end{enumerate}
In particular, \(E\) is a \((K_{\epsilon})^{>0}\)-graph.
\end{lemma}

\begin{proof}
By Proposition~\ref{prop:elementary-special-chains}, each forced chain is of
type {\rm (A)}--{\rm (D)} and can meet the rest of \(E\) only through its last
component.  Thus two forced chains with distinct first curves can interact only
through their terminal parts; 
if they share components, then they coincide from some component onward, toward
their last components.
Reading off the possible terminal vertices of the
four types gives exactly the configurations listed above, while all other
gluings are excluded by Theorem~\ref{thm:separatrix}.  The resulting
configurations are \((K_\epsilon)^{>0}\)-graphs by Lemma~\ref{lem:neg-def}.
\end{proof}

We may therefore assume, for the purpose of detecting new adjoint
configurations, that forced chains with distinct first curves are disjoint.  The
second way in which the reduction can close is residual: after the forced chains
have been peeled off, one has to analyze the remaining residual intersections.

We now peel off all maximal forced chains simultaneously.  Let
\(\Theta_1,\ldots,\Theta_s\) be the maximal special chains of type {\rm (A)},
and let \(\Xi_1,\ldots,\Xi_t\) be the maximal special chains of types
{\rm (B)}, {\rm (C)}, and {\rm (D)}.  Set
\[
K_{\epsilon,1}
:=
K_{\epsilon}
-\sum_{i=1}^sM(K_{\epsilon},\Theta_i)
-\sum_{j=1}^tM(K_{\epsilon},\Xi_j).
\]
This is the simultaneous version of the peeling construction of
Section~\ref{subsec:numerical-projection-peeling}, applied to the disjoint
maximal special chains.

\begin{lemma}[Simultaneous peeling]
\label{lem:K1-peeling}
The exceptional divisor \(E\) is a \((K_{\epsilon,1})^{\ge0}\)-graph.  Moreover, if \(C\) is
contained in one of the peeled chains, then
\(K_{\epsilon,1}\cdot C=0\).
\end{lemma}

\begin{proof}
The assertion \(K_{\epsilon,1}\cdot C=0\) on the peeled chains follows from
Lemma~\ref{lem:peeling-identity}.  The fact that \(E\) remains a
\((K_{\epsilon,1})^{\ge0}\)-graph follows from Lemma~\ref{lem:peeling-decomposition},
applied successively to the peeled chains.
\end{proof}

\subsection{Residual components and the closing condition}
\label{subsec:residual-components}

We now analyze the residual intersections left after the simultaneous peeling,
using the residual-cluster notation of Section~\ref{subsec:boundary-residual}.
Let \(C\) be a residual component, namely a component not contained in any peeled
chain.  Its residual intersection is obtained from \(K_\epsilon\cdot C\) by
subtracting the boundary contributions of the peeled chains attached to \(C\).

Suppose that \(C\) meets \(k\) peeled chains of type {\rm (A)}, with determinants
\(n_1,\ldots,n_k\), and \(l\) peeled chains of types {\rm (B)}, {\rm (C)}, and
{\rm (D)}, with determinants \(m_1,\ldots,m_l\).  Then
\[
K_{\epsilon,1}\cdot C
=
K_\epsilon\cdot C
-\sum_{i=1}^k\frac{1+\epsilon}{n_i}
-\sum_{j=1}^l\frac{\epsilon}{m_j}.
\]
These residual intersections are the remaining quantities to be controlled
after the forced chains have been peeled off.  The following two lemmas determine
all residual components for which this intersection is non-positive.

Let
\[
\mathcal R_C
=
C+\sum_{i=1}^k\Theta_i+\sum_{j=1}^l\Xi_j
\]
be the residual cluster centered at \(C\), and write
\[
\rho:=\rho_C.
\]
Then
\[
\rho=d(C)-k-l.
\]

If \(C\) is \(\mathcal F\)-invariant, set
\[
z:={\rm Z}(\mathcal F,C)-d_2(C),
\]
which is a non-negative integer.

\begin{lemma}[Non-invariant residual components]
\label{lem:non-inv-residual}
Let \(C\) be a non-\(\mathcal F\)-invariant residual component with
\(K_{\epsilon,1}\cdot C\le0\).
Then
\[
p_a(C)=0,\qquad \operatorname{tang}(\mathcal F,C)=0,\qquad d_1(C)=0.
\]
Moreover, \(k=0\).
The possible residual data are exactly:
\[
\begin{array}{c|c|c|c}
\rho & l & K_{\epsilon,1}\cdot C & \text{determinant type}\\
\hline
2 & 0 & 0 & \text{none}\\
1 & 2 & 0 & (2,2)\\
0 & 3 & <0 &
(2,2,n),\ n\ge2,\ (2,3,3),\ (2,3,4),\ (2,3,5)\\
0 & 3 & 0 &
(2,3,6),\ (2,4,4),\ (3,3,3)\\
0 & 4 & 0 & (2,2,2,2).
\end{array}
\]
Here the determinant types refer to the adjacent peeled chains, which are all of
type {\rm (C)}.  When \(\rho>0\), every exterior branch of \(\mathcal R_C\) is
\(\mathcal F\)-invariant.
\end{lemma}

\begin{proof}
By Proposition~\ref{prop:elementary-special-chains}, a non-\(\mathcal F\)-invariant residual component can meet a peeled special
chain only through a chain of type {\rm (C)}.  Hence \(k=0\), all adjacent
peeled chains are of type {\rm (C)} and have determinants \(m_j\ge2\).  Thus
\[
K_{\epsilon,1}\cdot C
=
\operatorname{tang}(\mathcal F,C)+d_1(C)
+\epsilon\left(
2p_a(C)-2+\rho+l-\sum_{j=1}^l\frac1{m_j}
\right).
\]

Since \(C\) is a residual component after the forced chains have been peeled off,
we have \(d(C)=\rho+l>0\).
Together with \(0<\epsilon<1/3\), elementary estimates show that
\(K_{\epsilon,1}\cdot C\le0\) forces
\(
p_a(C)=\operatorname{tang}(\mathcal F,C)=d_1(C)=0\).

Consequently \(d(C)=d_2(C)=\rho+l\), and
\begin{equation}\label{equ:K1C-non-inv-residual}
K_{\epsilon}\cdot C=\epsilon(\rho+l-2),\qquad
K_{\epsilon,1}\cdot C
=
\epsilon\left(\rho+l-2-\sum_{j=1}^l\frac1{m_j}\right).
\end{equation}

Since \(C\) is not the first curve of a special chain, \(K_{\epsilon}\cdot C\ge0\).  If
\(K_{\epsilon}\cdot C=0\) and \(l>0\), then adjoining \(C\) to an adjacent peeled chain
would contradict maximality.  Hence
\[
\rho+l-2\ge0,
\]
with equality only when \(l=0\).  On the other hand,
\(K_{\epsilon,1}\cdot C\le0\) and \(m_j\ge2\) give
\[
\rho+l-2\le \sum_{j=1}^l\frac1{m_j}\le \frac l2.
\]
Thus \(2\rho+l\le4\).  Together with \(\rho+l\ge2\) and the equality condition
above, this gives
\[
(\rho,l)=(2,0),(1,2),(0,3),(0,4).
\]
For \((\rho,l)=(2,0)\), no peeled chain is attached to \(C\), and
\eqref{equ:K1C-non-inv-residual} gives \(K_{\epsilon,1}\cdot C=0\).
For \((\rho,l)=(1,2)\), the inequality forces \((m_1,m_2)=(2,2)\) and $K_{\epsilon,1}\cdot C=0$.
For \((\rho,l)=(0,3)\), it becomes
\[
\sum_{j=1}^3\frac1{m_j}\ge1,
\]
giving the strict cases \((2,2,n)\), \(n\ge2\), \((2,3,3)\), \((2,3,4)\), and
\((2,3,5)\), and the equality cases \((2,3,6)\), \((2,4,4)\), and \((3,3,3)\).
Strictness in the inequality is equivalent to \(K_{\epsilon,1}\cdot C<0\).
For \((\rho,l)=(0,4)\), it forces
\((m_1,\dots,m_4)=(2,2,2,2)\), and in this case \(K_{\epsilon,1}\cdot C=0\).

This proves the table.
\end{proof}

\begin{lemma}[Invariant residual components]
\label{lem:inv-residual}
Let \(C\) be an \(\mathcal F\)-invariant residual component with
\(K_{\epsilon,1}\cdot C\le0\).  Then \(C\) is smooth rational.  The possible residual data
are exactly:
\[
\begin{array}{c|c|c|c|c|c}
\rho & k & l & z & K_{\epsilon,1}\cdot C & \text{determinant type}\\
\hline
2 & 0 & 0 & 0 & 0 & \text{none}\\
1 & 2 & 0 & 0 & 0 & (2,2)\\
0 & 2 & 0 & 1 & <0 & (2,2)\\
0 & 2 & 0 & 1 & \le0 & (2,3)\\
0 & 2 & 1 & 0 & <0 & (2,2;\,m),\ m\ge2 .
\end{array}
\]
In the last row, \(m=\det(-\Xi_1)\), where \(\Xi_1\) is the additional
type {\rm (B)} chain counted by \(l=1\). 
The row containing determinant type \((2,3)\) can occur only when
\(\epsilon\ge1/5\).  We also record the following data for the rows with
\(\rho>0\): in the row \((2,0,0,0)\), \(\mathcal R_C=C\),
\({\rm Z}(\mathcal F,C)=d_2(C)\), and
\((d_1(C),d_2(C))=(1,1)\) or \((0,2)\); in the row \((1,2,0,0)\),
\(d(C)=d_2(C)={\rm Z}(\mathcal F,C)=3\).
\end{lemma}

\begin{proof}
By the boundary contribution formula, we have
\[
K_{\epsilon,1}\cdot C
=
K_{\epsilon}\cdot C-\sum_{i=1}^k\frac{1+\epsilon}{n_i}
  -\sum_{j=1}^l\frac{\epsilon}{m_j}.
\]
For an invariant component,
\[
K_{\epsilon}\cdot C=
(1+\epsilon)(2p_a(C)-2+d(C))
+{\rm Z}(\mathcal F,C)-d_2(C).
\]
By the notation fixed above, \(d(C)=\rho+k+l\) and
\(
z={\rm Z}(\mathcal F,C)-d_2(C).
\)
Then we get
\[
K_{\epsilon,1}\cdot C
=
(1+\epsilon)(2p_a(C)-2+\rho+k+l)+z
-\sum_{i=1}^k\frac{1+\epsilon}{n_i}
-\sum_{j=1}^l\frac{\epsilon}{m_j}.
\]
Using \(n_i\ge2\) and \(m_j\ge1\), this gives
\[
K_{\epsilon,1}\cdot C
\ge
(1+\epsilon)(2p_a(C)-2+\rho)
+z+\frac{1+\epsilon}{2}k+l.
\]
Since \(d(C)=\rho+k+l>0\), \(z\ge0\), and \(0<\epsilon<1/3\), elementary
estimates show that \(K_{\epsilon,1}\cdot C\le0\) forces
\(p_a(C)=0\).
Thus \(C\) is smooth rational, and
\[
K_{\epsilon}\cdot C=(1+\epsilon)(\rho+k+l-2)+z.
\]
Moreover,
\begin{equation}\label{equ:K1C-inv-residual}
K_{\epsilon,1}\cdot C
=
(1+\epsilon)(\rho+k+l-2)+z
-\sum_{i=1}^k\frac{1+\epsilon}{n_i}
-\sum_{j=1}^l\frac{\epsilon}{m_j}.
\end{equation}

Since \(C\) is not the first curve of a special chain, \(K_{\epsilon}\cdot C\ge0\).  If
\(K_{\epsilon}\cdot C=0\) and \(k+l>0\), then adjoining \(C\) to an adjacent peeled chain
would contradict maximality.  Hence
\[
(1+\epsilon)(\rho+k+l-2)+z\ge0,
\]
with equality only when \(k=l=0\).  Moreover, the separatrix theorem (Theorem~\ref{thm:separatrix}) gives
\[
\rho+l+z\ge1.
\]
Indeed, if \(\rho+l+z=0\), then \(\rho=l=z=0\), so
\({\rm Z}(\mathcal F,C)=d_2(C)\).  Since \(C\) meets exactly \(k\) type
{\rm (A)} chains, we have \(k\le d_2(C)\le d(C)=k\), hence
\({\rm Z}(\mathcal F,C)=d_2(C)=k\).  Thus the closed cluster
\(\mathcal R_C=C+\Theta_1+\cdots+\Theta_k\) has no separatrix leaving it,
contradicting Theorem~\ref{thm:separatrix}.

On the other hand, \(K_{\epsilon,1}\cdot C\le0\), \(n_i\ge2\), and \(m_j\ge1\) imply
\[
(1+\epsilon)\left(\rho+\frac{k}{2}-2\right)+l+z\le0.
\]
Solving these elementary inequalities in non-negative integers, together with
the equality condition above and the assumption \(0<\epsilon<1/3\), gives
\[
(\rho,k,l,z)=(2,0,0,0),(1,2,0,0),(0,2,0,1),(0,2,1,0).
\]

It remains to determine the determinant types. 

For \((\rho,k,l,z)=(2,0,0,0)\), no peeled chain is attached to \(C\), so
\(\mathcal R_C=C\), and \eqref{equ:K1C-inv-residual} gives
\(K_{\epsilon,1}\cdot C=0\).  Here \(d(C)=2\) and
\({\rm Z}(\mathcal F,C)=d_2(C)\in\{1,2\}\).  Hence \(C\) is either a
\((-1)\)-\(\mathcal F\)-curve with \(d_1(C)=d_2(C)=1\), or a
\((-2)\)-\(\mathcal F\)-curve with \(d(C)=d_2(C)=2\).

For \((\rho,k,l,z)=(1,2,0,0)\), the inequality forces
\((n_1,n_2)=(2,2)\), and hence \(K_{\epsilon,1}\cdot C=0\).  Thus \(C\)
meets two type {\rm (A)} chains of determinant \(2\).  Since \(d(C)=3\) and 
Theorem~\ref{thm:separatrix} gives \({\rm Z}(\mathcal F,C)\ge k+1=3\), we get
\(d(C)=d_2(C)={\rm Z}(\mathcal F,C)=3\).

For \((\rho,k,l,z)=(0,2,0,1)\), the inequality becomes
\[
\frac1{n_1}+\frac1{n_2}\ge\frac1{1+\epsilon}.
\]
Thus \((n_1,n_2)=(2,2)\) or \((2,3)\).  The first case gives
\(K_{\epsilon,1}\cdot C<0\).  In the second case,
\begin{equation}\label{ineq:lc-wall-origin}
K_{\epsilon,1}\cdot C
=
1-(1+\epsilon)\left(\frac12+\frac13\right)
=
\frac{1-5\epsilon}{6},
\end{equation}
so this row can occur only when \(\epsilon\ge1/5\).

For \((\rho,k,l,z)=(0,2,1,0)\), the additional peeled chain \(\Xi_1\) is
necessarily of type {\rm (B)}. 
Indeed, the type {\rm (D)} possibility would give a closed residual cluster with
no separatrix leaving it, contradicting Theorem~\ref{thm:separatrix}.
Hence
\[
K_{\epsilon,1}\cdot C
=
1+\epsilon
-(1+\epsilon)\left(\frac1{n_1}+\frac1{n_2}\right)
-\frac{\epsilon}{m_1}.
\]
Thus \((n_1,n_2)=(2,2)\) or \((2,3)\).
If \((n_1,n_2)=(2,2)\), then
\[
K_{\epsilon,1}\cdot C
=
-\frac{\epsilon}{m_1}<0.
\]
The case \(m_1=1\) is excluded by
Remark~\ref{rem:determinant-one-forced-chain}: the external separatrix of
\(\Xi_1\) is forced to be \(C\).  Hence the closed residual cluster
\(
\mathcal R_C=C+\Theta_1+\Theta_2+\Xi_1
\)
has no separatrix outside it, contradicting Theorem~\ref{thm:separatrix}.
Therefore this gives only the row \((2,2;\,m_1)\) with \(m_1\ge2\).
If \((n_1,n_2)=(2,3)\), then
\[
K_{\epsilon,1}\cdot C
=
\frac{1-5\epsilon}{6}
+\epsilon\left(1-\frac1{m_1}\right).
\]
In the range \(0<\epsilon<1/3\), the inequality
\(K_{\epsilon,1}\cdot C\le0\) forces \(m_1=1\).  But this determinant-one case
has just been excluded.  Hence no row of type \((2,3;\,1)\) occurs.

The rows with \(\rho>0\) are \(K_{\epsilon,1}\)-trivial, and strict negativity can occur only
when \(\rho=0\).  This proves the table.
\end{proof}

\begin{remark}[The log canonical wall \(1/5\)]\label{rem:lc-wall-origin}
The identity \eqref{ineq:lc-wall-origin} is the numerical origin of the wall
\(1/5\).  After the forced special chains have been peeled off, the new boundary
rows are precisely the \((2,3)\)-rows, and their residual intersection is
\[
\frac{1-5\epsilon}{6}.
\]
Thus these rows are excluded for \(0<\epsilon<1/5\).  At
\(\epsilon=1/5\), the corresponding residual intersection becomes zero, so the
boundary pattern of determinant type \((2,3)\) first enters the admissible list.
This is the wall \(1/5\).
\end{remark}

\begin{corollary}[Strict residual negativity forces closure]
\label{cor:strict-residual-closed}
Let \(C\) be a residual component.  If
\(K_{\epsilon,1}\cdot C<0\),
then
\(\rho_C=0\).
Equivalently, the residual cluster \(\mathcal R_C\) is closed in \(E\).  Since
\(E\) is connected, this means
\[
\mathcal R_C=E.
\]
\end{corollary}

\begin{proof}
This follows immediately from Lemmas~\ref{lem:non-inv-residual} and
\ref{lem:inv-residual}.  In both lists, the rows with positive residual valency
are \(K_{\epsilon,1}\)-trivial, while strict negativity can occur only when \(\rho=0\).
\end{proof}

\begin{corollary}[Residual alternatives after peeling]
\label{cor:residual-alternatives}
After the simultaneous peeling, one of the following holds:
\begin{enumerate}
\item there exists a residual component \(C\) such that
\(K_{\epsilon,1}\cdot C<0\),
and then
\(E=\mathcal R_C\);
\item no residual component has negative \(K_{\epsilon,1}\)-intersection, and then
\(E\) is a \((K_{\epsilon,1})^{=0}\)-graph.
\end{enumerate}
\end{corollary}

\begin{proof}
If there is a residual component \(C\) with
\(K_{\epsilon,1}\cdot C<0\), then
Corollary~\ref{cor:strict-residual-closed} gives
\(\mathcal R_C=E\).

Assume now that no residual component has negative
\(K_{\epsilon,1}\)-intersection.  The peeled components are
\(K_{\epsilon,1}\)-trivial by Lemma~\ref{lem:K1-peeling}.  Hence every component
of \(E\) has non-negative \(K_{\epsilon,1}\)-intersection.  Since \(E\) is a
\((K_{\epsilon,1})^{\ge0}\)-graph, the negativity criterion forces
\(M(K_{\epsilon,1},E)=0\).
Equivalently, \(E\) is a \((K_{\epsilon,1})^{=0}\)-graph.
\end{proof}

\subsection{The wall \(1/5\) and the classification theorem}
\label{subsec:classification}

We now assemble the classification.  By Corollary~\ref{cor:residual-alternatives},
after peeling all maximal special chains, either a closed residual cluster remains,
or the whole configuration is \(K_{\epsilon,1}\)-trivial.  Together with the
interaction alternatives of Lemma~\ref{lem:interaction}, this exhausts
all possibilities.

\begin{theorem}[Classification of \(\epsilon\)-adjoint log canonical singularities]
\label{thm:adjoint-lc-classification}
Fix \(\epsilon\in(0,1/3)\).  Let \((Y,\mathcal G,p)\) be an
\(\epsilon\)-adjoint log canonical foliated surface singularity, and let
\[
\pi:(X,\mathcal F)\to(Y,\mathcal G)
\]
be the minimal resolution with reduced exceptional divisor \(E\).  Then \(E\)
is one of the following configurations.

\medskip
\noindent
{\rm\bf (I) Foliated canonical configurations.}
\begin{enumerate}
\item A chain of smooth rational \(\mathcal F\)-invariant curves, consisting of
one of the following:
\begin{enumerate}
\item an \(\mathcal F\)-chain (Definition~\ref{def:F-chain});
\item two \((-1)\)-\(\mathcal F\)-curves of self-intersection \(-2\) joined by
a bad tail (Definition~\ref{def:bad-tail});
\item a chain of \((-2)\)-\(\mathcal F\)-curves
(Definition~\ref{def:-2Fcurve}).
\end{enumerate}

\item An \(\mathcal F\)-dihedral fork (Definition~\ref{def:F-dihedral-fork}).

\item An elliptic Gorenstein leaf (Definition~\ref{def:egl}), namely either a rational
\(\mathcal F\)-invariant curve with a single node, or a cycle of
\((-2)\)-\(\mathcal F\)-curves.
\end{enumerate}

\medskip
\noindent
{\rm\bf (II) Foliated log canonical non-canonical configurations.}
\begin{enumerate}
\item An \(\mathcal F\)-star chain (Definition~\ref{def:F-star-graph}) centered at a smooth rational
non-\(\mathcal F\)-invariant curve \(C\) with
\(\operatorname{tang}(\mathcal F,C)=0\).

\item An \(\mathcal F\)-star fork (Definition~\ref{def:F-star-graph}) centered at a smooth rational
non-\(\mathcal F\)-invariant curve \(C\) with
\(\operatorname{tang}(\mathcal F,C)=0\).
Its determinant type is one of the following:
\begin{enumerate}
\item $(2,2,n)$, $n\ge2$, $(2,3,3)$, $(2,3,4)$, $(2,3,5)$;
\item $(2,3,6)$, $(2,4,4)$, $(3,3,3)$, $(2,2,2,2)$.
\end{enumerate}

\item A smooth elliptic non-\(\mathcal F\)-invariant curve \(C\) with
\(\operatorname{tang}(\mathcal F,C)=0\).
\end{enumerate}

\medskip
\noindent
{\rm \bf (III) Boundary configurations.}
There is a smooth rational \(\mathcal F\)-invariant curve \(C\) with
\({\rm Z}(\mathcal F,C)=3\).  The exceptional divisor is of the form
\[
E=C+\Theta_1+\Theta_2,
\]
where \(\Theta_1,\Theta_2\) are \(\mathcal F\)-chains
(Definition~\ref{def:F-chain}) attached to \(C\), with determinant type
\((2,3)\).
\medskip

Moreover, the following properties hold.
\begin{enumerate}
\item Type {\rm (I)} is foliated canonical, Type {\rm (II)} is foliated log
canonical but not foliated canonical, and Type {\rm (III)} is not foliated log
canonical.  In particular, Type {\rm (III)} cannot occur if
\[
0<\epsilon<\frac15.
\]

\item Among Types {\rm (I)} and {\rm (II)}, the \(\epsilon\)-adjoint klt cases are
precisely ${\rm (I\mbox{-}1a)}$, ${\rm (I\mbox{-}1b)}$, ${\rm (I\mbox{-}1c)}$,
${\rm (I\mbox{-}2)}$, ${\rm (II\mbox{-}1)}$, ${\rm (II\mbox{-}2a)}$.
The cases
${\rm (I\mbox{-}3)}$, ${\rm (II\mbox{-}2b)}$, ${\rm (II\mbox{-}3)}$
are \(\epsilon\)-adjoint log canonical but not \(\epsilon\)-adjoint klt.
\end{enumerate}
\end{theorem}

\begin{proof}
If \(E\) is irreducible or \(K_{\epsilon}\)-trivial, the conclusion follows from
Lemma~\ref{lem:base-cases} and the local lists.  These cases, including
Types {\rm (I\mbox{-}3)} and {\rm (II\mbox{-}3)}, are contained in
Types {\rm (I)} and {\rm (II)}.

Assume now that \(E\) is neither irreducible nor \(K_{\epsilon}\)-trivial.  Then
there exists a \(K_{\epsilon}\)-negative vertex, and it generates a special
\((K_{\epsilon})^{>0}\)-chain.  If two such chains with distinct first curves
interact, Lemma~\ref{lem:interaction} gives an \(\mathcal F\)-chain, a chain of
\((-2)\)-\(\mathcal F\)-curves, or an \(\mathcal F\)-star chain.  These give
Types {\rm (I\mbox{-}1a)}, {\rm (I\mbox{-}1c)}, and {\rm (II\mbox{-}1)}.

Thus we may assume that the maximal special chains with distinct first curves
are disjoint.  We peel them off simultaneously and obtain \(K_{\epsilon,1}\).
By Lemma~\ref{lem:K1-peeling}, \(E\) remains a
\((K_{\epsilon,1})^{\ge0}\)-graph, and the peeled components are
\(K_{\epsilon,1}\)-trivial.

By Corollary~\ref{cor:residual-alternatives}, there are two possibilities after
the simultaneous peeling.  First, there may exist a residual component \(C\) with
\[
K_{\epsilon,1}\cdot C<0.
\]
Then \(E=\mathcal R_C\), so only the rows with \(\rho=0\) can occur.  The rows
with \(\rho=0\) in Lemmas~\ref{lem:non-inv-residual} and
\ref{lem:inv-residual} give the usual foliated log canonical configurations: the
non-invariant rows give the \(\mathcal F\)-star forks of Type
{\rm (II\mbox{-}2a)}, the invariant row of determinant type \((2,2)\) gives Type
{\rm (I\mbox{-}1b)}, and the invariant rows of determinant type
\((2,2;\,m)\), \(m\ge2\), give Type {\rm (I\mbox{-}2)}.  The only remaining row
with \(\rho=0\), of determinant type \((2,3)\), is precisely the new boundary
configuration of Type {\rm (III)}.

Second, no residual component has negative \(K_{\epsilon,1}\)-intersection.  By
Corollary~\ref{cor:residual-alternatives}, \(E\) is a
\((K_{\epsilon,1})^{=0}\)-graph.  Hence every residual component lies in a
\(K_{\epsilon,1}\)-trivial row of Lemma~\ref{lem:non-inv-residual} or
Lemma~\ref{lem:inv-residual}.  The \(K_{\epsilon,1}\)-trivial rows with
\(\rho=0\) give the complete residual configurations: the non-invariant rows
give the star-fork cases of Type {\rm (II\mbox{-}2b)}, while the invariant
\((2,3)\)-row gives Type {\rm (III)} at the wall \(\epsilon=1/5\).

It remains to exclude the \(K_{\epsilon,1}\)-trivial rows with \(\rho>0\).
Since \(\mathcal R_C\) already contains all peeled chains meeting its center,
an exterior branch of \(\mathcal R_C\) can meet only another residual component.
Using the branch data recorded in Lemmas~\ref{lem:non-inv-residual} and
\ref{lem:inv-residual}, any closed assembly made from such residual clusters
would contain a closed invariant subcluster with no separatrix leaving it,
contradicting Theorem~\ref{thm:separatrix}.

The identification of Types {\rm (I)} and {\rm (II)} with the usual foliated
canonical and foliated log canonical configurations follows from the standard
classification recalled in Theorem~\ref{thm:foliated-lc-classification}.  Type
{\rm (III)} is the new boundary type, and is not foliated log canonical.  The
threshold statement for Type {\rm (III)} follows from
Lemma~\ref{lem:inv-residual}, or equivalently from
Remark~\ref{rem:lc-wall-origin}: these boundary configurations occur only at or
beyond the wall \(\epsilon=1/5\).

Finally, the \(\epsilon\)-adjoint klt assertion follows from the strict version
of the same numerical criterion.  Since the \(K_{\epsilon,1}\)-trivial rows with
\(\rho>0\) have been excluded, each configuration in the list arises from a
unique case of the reduction: a base case, an interaction case, or a row with
\(\rho=0\).  Reading whether the corresponding numerical projection is strictly
positive gives exactly the klt cases stated in (2), and the remaining listed
cases are log canonical but not klt.  This proves the theorem.
\end{proof}

\begin{remark}
\label{rem:wall-one-fifth-classification}
The classification is the output of the negative-vertex reduction, not a direct
enumeration of dual graphs.  As explained in Remark~\ref{rem:lc-wall-origin},
the wall \(1/5\) is the first value at which the residual analysis admits the
determinant type \((2,3)\) boundary pattern.
\end{remark}

\section{Adjoint canonical singularities and the wall \(1/4\)}
\label{sec:adjoint-canonical}

In this section we study the adjoint canonical condition.  Let
\[
\pi:(X,\mathcal F)\to(Y,\mathcal G)
\]
be the minimal resolution, and let \(E\) be the reduced exceptional divisor.  Set
\[
A_{\epsilon}:=K_{\mathcal F}+\epsilon K_X.
\]
By Definition~\ref{def:adj-lc-sing}, \(p\in(Y,\mathcal G)\) is \(\epsilon\)-adjoint canonical if and only if
\[
E\text{ is a }(A_{\epsilon})^{\ge0}\text{-graph}.
\]
Thus the adjoint canonical classification is obtained by applying an
\(A_{\epsilon}\)-effectivity filter to the configurations in
Theorem~\ref{thm:adjoint-lc-classification}.  In this sense, the wall \(1/4\)
is not obtained by a new enumeration of graphs, but by imposing an additional
effectivity test on the adjoint log canonical list.
\medskip

For an irreducible component \(C\subset E\), we shall use the following local
formulas.  If \(C\) is non-\(\mathcal F\)-invariant, then
\[
A_{\epsilon}\cdot C
=
\operatorname{tang}(\mathcal F,C)
+\epsilon(2p_a(C)-2)
-(1+\epsilon)C^2.
\]
If \(C\) is \(\mathcal F\)-invariant, then
\[
A_{\epsilon}\cdot C
=
2p_a(C)-2+{\rm Z}(\mathcal F,C)
+\epsilon(2p_a(C)-2-C^2).
\]
Both formulas follow directly from the tangency formula
(Proposition~\ref{prop:tang-index}), the \(Z\)-index formula
(Proposition~\ref{prop:Z-index-and-CS-index}), and adjunction.

\subsection{The adjoint canonical effectivity filter}
\label{subsec:Kprime-effectivity-test}

We first remove the configurations which cannot be \(\epsilon\)-adjoint canonical.

\begin{lemma}[Excluded configurations]
\label{lem:canonical-excluded-configurations}
Assume \(0<\epsilon<1/4\).  The following configurations in
Theorem~\ref{thm:adjoint-lc-classification} are not \(\epsilon\)-adjoint
canonical:
\[
{\rm (I\mbox{-}3)},\qquad
{\rm (II\mbox{-}2b)},\qquad
{\rm (II\mbox{-}3)},\qquad
{\rm (II\mbox{-}2a)},\qquad
{\rm (III)}.
\]
\end{lemma}

\begin{proof}
The first three configurations are already non-klt in the
\(\epsilon\)-adjoint log canonical classification
(Theorem~\ref{thm:adjoint-lc-classification}), and hence cannot be
\(\epsilon\)-adjoint canonical.

We next exclude Type {\rm (II\mbox{-}2a)}.  Let
\(
E=C+\Theta_1+\Theta_2+\Theta_3
\)
be of Type {\rm (II\mbox{-}2a)}, centered at a smooth rational
non-\(\mathcal F\)-invariant curve \(C\) with
\({\rm tang}(\mathcal F,C)=0\).  Since the intersection matrix is negative
definite, \(C^2\le -2\).  Thus
\[
A_{\epsilon}\cdot C
=
-2\epsilon-(1+\epsilon)C^2
\ge 2.
\]

For each attached chain \(\Theta_i\), let
\(\Theta'_i\subseteq\Theta_i\) be the maximal special
\((A_\epsilon)^{>0}\)-subchain generated from the component of \(\Theta_i\)
meeting \(C\), if it exists; otherwise set \(\Theta'_i=0\).  Put
\[
A_{\epsilon,1}
:=
A_\epsilon-\sum_i M(A_\epsilon,\Theta'_i),
\]
omitting the summand when \(\Theta'_i=0\).  If \(E\) were an
\((A_\epsilon)^{\ge0}\)-graph, then by the peeling decomposition \(E\) would
also be an \((A_{\epsilon,1})^{\ge0}\)-graph.  By maximality and the extension
criterion, all components different from \(C\) have non-negative
\(A_{\epsilon,1}\)-intersection.

It remains to estimate \(A_{\epsilon,1}\cdot C\).  The determinant triples in
Type {\rm (II\mbox{-}2a)} are
\[
(2,2,n),\ n\ge2,\qquad (2,3,m),\ m=3,4,5.
\]
For each \(i\), the boundary contribution of \(\Theta'_i\) to \(C\) is bounded
above by the determinant contribution of the whole attached
\(\mathcal F\)-chain:
\[
M(A_\epsilon,\Theta'_i)\cdot C\le \frac{n_i-1}{n_i};
\]
if \(\Theta'_i=0\), the left-hand side is understood to be \(0\).  Hence
\[
A_{\epsilon,1}\cdot C
\ge
2-\sum_{i=1}^3\frac{n_i-1}{n_i}.
\]
For the triples \((2,2,n)\), \(n\ge2\), the right-hand side equals
\[
2-\frac12-\frac12-\frac{n-1}{n}
=
\frac1n>0.
\]
For the triples \((2,3,m)\), \(m=3,4,5\), it is at least
\[
2-\frac12-\frac23-\frac45
=
\frac1{30}>0.
\]
Thus \(A_{\epsilon,1}\cdot C>0\), while all other components have
non-negative \(A_{\epsilon,1}\)-intersection.  This contradicts the
\((A_{\epsilon,1})^{\ge0}\)-effectivity on the negative definite configuration.
Hence Type {\rm (II\mbox{-}2a)} is excluded.

We now exclude Type {\rm (III)}.  Let \(C\) be the invariant center with
\({\rm Z}(\mathcal F,C)=3\), and let \(\Theta_1,\Theta_2\) be the two attached
\(\mathcal F\)-chains of determinant type \((2,3)\).  After relabeling, assume
\(\det(-\Theta_1)=2\) and \(\det(-\Theta_2)=3\).

Since \(\det(-\Theta_1)=2\), the chain \(\Theta_1\) consists of a single
\((-2)\)-curve \(\Gamma\) meeting \(C\).  It satisfies
\(A_\epsilon\cdot\Gamma=-1\), and hence
\[
M(A_\epsilon,\Theta_1)=\frac12\Gamma,
\qquad
M(A_\epsilon,\Theta_1)\cdot C=\frac12.
\]
For the determinant \(3\) chain \(\Theta_2\), there are two possibilities.
\begin{itemize}
\item[\rm(i)] If \(\Theta_2=\Gamma'\) is a single \((-3)\)-curve meeting \(C\),
then \(A_\epsilon\cdot\Gamma'=-(1-\epsilon)\), and
\[
M(A_\epsilon,\Theta_2)
=
\frac{1-\epsilon}{3}\Gamma',
\qquad
M(A_\epsilon,\Theta_2)\cdot C
=
\frac{1-\epsilon}{3}.
\]

\item[\rm(ii)] If \(\Theta_2=\Gamma'_1+\Gamma'_2\) is the chain of two
\((-2)\)-curves, where \(\Gamma'_2\) is the component meeting \(C\), then
\(A_\epsilon\cdot\Gamma'_1=-1\) and \(A_\epsilon\cdot\Gamma'_2=0\).
Hence
\[
M(A_\epsilon,\Theta_2)
=
\frac23\Gamma'_1+\frac13\Gamma'_2,
\qquad
M(A_\epsilon,\Theta_2)\cdot C
=
\frac13.
\]
\end{itemize}
In both cases, \(\Theta_1\) and \(\Theta_2\) are \((A_\epsilon)^{>0}\)-graphs.
Set
\[
A_{\epsilon,1}
:=
A_{\epsilon}
-
M(A_{\epsilon},\Theta_1)
-
M(A_{\epsilon},\Theta_2).
\]
If \(E\) were an \((A_{\epsilon})^{\ge0}\)-graph, then \(E\) would be an
\((A_{\epsilon,1})^{\ge0}\)-graph, and all components different from \(C\)
would have non-negative \(A_{\epsilon,1}\)-intersection.  Hence it suffices to
show \(A_{\epsilon,1}\cdot C>0\).

In case {\rm (i)}, since \(C^2\le -1\),
\[
A_\epsilon\cdot C
=
1-2\epsilon-\epsilon C^2
\ge
1-\epsilon.
\]
Thus
\begin{equation}\label{ineq:K'C-1/4}
A_{\epsilon,1}\cdot C
\ge
1-\epsilon-\frac12-\frac{1-\epsilon}{3}
=
\frac{1-4\epsilon}{6}>0.
\end{equation}
In case {\rm (ii)}, we have \(C^2\le -2\).  Indeed, if \(C^2=-1\), then the
subconfiguration \(\Gamma+C+\Gamma'_2\) has determinant \(0\), contradicting the
negative definiteness of \(E\).  Hence \(A_\epsilon\cdot C\ge1\), and
\[
A_{\epsilon,1}\cdot C
\ge
1-\frac12-\frac13
=
\frac16>0.
\]
Thus Type {\rm (III)} is also excluded.
\end{proof}

\begin{remark}[The canonical wall \(1/4\)]
\label{rem:canonical-wall-origin}
The inequality \eqref{ineq:K'C-1/4} is the canonical analogue of
Remark~\ref{rem:lc-wall-origin}.  In the single \((-3)\)-chain case, after
imposing the \(A_{\epsilon}\)-effectivity filter, the residual intersection at
the center is bounded below by
\[
\frac{1-4\epsilon}{6},
\]
with equality if and only if \(C^2=-1\).  Hence Type {\rm (III)} is excluded
for \(0<\epsilon<1/4\), and this lower bound vanishes exactly at
\(\epsilon=1/4\).
\end{remark}

We now record the restrictions on the configurations which survive this
exclusion.

\begin{lemma}[Surviving configurations]
\label{lem:canonical-surviving-configurations}
Assume \(0<\epsilon<1/4\).  Among the remaining configurations in
Theorem~\ref{thm:adjoint-lc-classification}, if \(E\) is an
\((A_{\epsilon})^{\ge0}\)-graph, then the following restrictions hold.
\begin{enumerate}
\item If \(E=C_1+\cdots+C_r\) is a chain of
\((-2)\)-\(\mathcal F\)-curves of Type {\rm (I\mbox{-}1c)}, then exactly one
of the following occurs:
\begin{enumerate}
\item every component has self-intersection \(-2\).  In this case \(p\in Y\)
is a Du Val surface singularity, and the singularity is
\(\epsilon\)-adjoint canonical but not \(\epsilon\)-adjoint terminal;

\item \(E\) contracts to a smooth surface point.  In this case the induced
foliation \(\mathcal G\) is canonical at \(p\), and the singularity is
\(\epsilon\)-adjoint terminal.
\end{enumerate}

\item If \(E\) is a bad-tail chain of Type {\rm (I\mbox{-}1b)} or an
\(\mathcal F\)-dihedral fork of Type {\rm (I\mbox{-}2)}, then every component
of \(E\) has self-intersection \(-2\).

\item If \(E\) is an \(\mathcal F\)-star chain of Type {\rm (II\mbox{-}1)},
centered at the unique non-\(\mathcal F\)-invariant curve \(C\), then
\(C^2=-1\), and \(E\) has at least two components.
\end{enumerate}
\end{lemma}

\begin{proof}
(1) For a \((-2)\)-\(\mathcal F\)-curve \(C_i\), the local formula gives
\[
A_{\epsilon}\cdot C_i=\epsilon(-2-C_i^2).
\]
If \(C_i^2\le -2\) for all \(i\), then \(A_{\epsilon}\cdot C_i\ge0\) for every
\(i\).  Since \(E\) is negative definite and an
\((A_{\epsilon})^{\ge0}\)-graph, the negativity criterion forces
\(M(A_\epsilon,E)=0\).  Hence \(A_\epsilon\cdot C_i=0\) for every \(i\), so
\(C_i^2=-2\) for all \(i\).  This gives the Du Val case.

Suppose now that some component has self-intersection \(-1\).  Such a component
must be an end component; otherwise it would be \(\mathcal F\)-exceptional, and
there is at most one such component.  Contracting the maximal tail ending at
this \((-1)\)-curve reduces the remaining chain to the previous case.
Equivalently, up to reversing the chain,
\[
(-C_1^2,\ldots,-C_r^2)=(1,2,\ldots,2,3,2,\ldots,2).
\]
The contraction is smooth on the surface, the induced foliation is canonical,
and the numerical projection is strictly effective.  Hence the singularity is
\(\epsilon\)-adjoint terminal.

(2) Consider a bad-tail chain or an \(\mathcal F\)-dihedral fork.  Let \(C\) be
the bad tail or the fork center, and let \(\Gamma_1,\Gamma_2\) be the two
adjacent \((-1)\)-\(\mathcal F\)-curves of self-intersection \(-2\).  The center
satisfies \(C^2\le -2\), and the remaining components, if any, are
\((-2)\)-\(\mathcal F\)-curves with self-intersection at most \(-2\).  Since
\[
M(A_{\epsilon},\Gamma_i)=\frac12\Gamma_i,\qquad i=1,2,
\]
set
\[
A_{\epsilon,1}:=A_{\epsilon}-\frac12\Gamma_1-\frac12\Gamma_2.
\]
If \(E\) is an \((A_{\epsilon})^{\ge0}\)-graph, then \(E\) is an
\((A_{\epsilon,1})^{\ge0}\)-graph.  Moreover,
\[
A_{\epsilon,1}\cdot C=\epsilon(-2-C^2),
\qquad
A_{\epsilon,1}\cdot D=\epsilon(-2-D^2)
\]
for every remaining component \(D\), and
\(A_{\epsilon,1}\cdot\Gamma_i=0\).  Thus every component of \(E\) has
non-negative \(A_{\epsilon,1}\)-intersection.  The negativity criterion forces
\(M(A_{\epsilon,1},E)=0\).  Hence all these intersections vanish, and every
component of \(E\) has self-intersection \(-2\).  The resulting surface
singularity is Du Val, and the numerical projection is non-strict; hence the
singularity is not \(\epsilon\)-adjoint terminal.

(3) Let \(E\) be an \(\mathcal F\)-star chain centered at the
non-\(\mathcal F\)-invariant curve \(C\).  Write
\[
E=C+\Theta_1+\cdots+\Theta_l,\qquad l\le2,
\]
where the \(\Theta_i\) are the attached \(\mathcal F\)-chains.  For each \(i\),
let \(\Theta'_i\subseteq\Theta_i\) be the maximal subchain, starting from the
component meeting \(C\), such that \(\Theta'_i\) is an
\((A_\epsilon)^{\ge0}\)-graph.  Put
\[
A_{\epsilon,1}
:=
A_\epsilon-\sum_{i=1}^l M(A_\epsilon,\Theta'_i),
\]
omitting the summand when \(\Theta'_i=0\).  Then \(E\) is an
\((A_{\epsilon,1})^{\ge0}\)-graph, and all components different from \(C\) have
non-negative \(A_{\epsilon,1}\)-intersection.

Assume \(C^2\le -2\).  Since \({\rm tang}(\mathcal F,C)=0\),
\[
A_\epsilon\cdot C=-2\epsilon-(1+\epsilon)C^2\ge2.
\]
Each nonzero \(\Theta'_i\) contributes strictly less than \(1\) to \(C\).  Since
\(l\le2\), we get
\[
A_{\epsilon,1}\cdot C
=
A_\epsilon\cdot C-\sum_{i=1}^l M(A_\epsilon,\Theta'_i)\cdot C
>
2-l
\ge0.
\]
Thus every component has non-negative \(A_{\epsilon,1}\)-intersection, and the
intersection at \(C\) is positive.  This contradicts the
\((A_{\epsilon,1})^{\ge0}\)-effectivity on the negative definite configuration.
Therefore \(C^2=-1\).

If no chain is attached to \(C\), then \(E=C\) and
\[
A_\epsilon\cdot C=1-\epsilon>0,
\]
which contradicts the \((A_\epsilon)^{\ge0}\)-effectivity on the negative
definite curve \(C\).  Hence at least one \(\mathcal F\)-chain is attached to
\(C\), and \(E\) has at least two components.
\end{proof}

\subsection{Classification below \(1/4\)}
\label{subsec:canonical-classification}

\begin{theorem}[Classification of adjoint canonical singularities below \(1/4\)]
\label{thm:adjoint-canonical-classification}
Fix $\epsilon\in(0,1/4)$.
Let \(p\in(Y,\mathcal G)\) be an \(\epsilon\)-adjoint canonical foliated surface
singularity, and let
\[
\pi:(X,\mathcal F)\to(Y,\mathcal G)
\]
be the minimal resolution with reduced exceptional divisor \(E\).  Then \(E\) is
one of the following configurations:
\begin{enumerate}
\item an \(\mathcal F\)-chain (Definition~\ref{def:F-chain});

\item a chain of \((-2)\)-\(\mathcal F\)-curves
(Definition~\ref{def:-2Fcurve});

\item two \((-1)\)-\(\mathcal F\)-curves of self-intersection \(-2\) joined by a
bad tail (Definition~\ref{def:bad-tail}).  Moreover, the bad tail has
self-intersection \(-2\);

\item an \(\mathcal F\)-dihedral fork (Definition~\ref{def:F-dihedral-fork}), all of whose
components have self-intersection \(-2\);

\item a non-trivial \(\mathcal F\)-star chain
(Definition~\ref{def:F-star-graph}) centered at a smooth rational
non-\(\mathcal F\)-invariant curve \(C\) with
\(\operatorname{tang}(\mathcal F,C)=0\) and \(C^2=-1\).
\end{enumerate}

Moreover:
\begin{enumerate}
\item[\rm (i)] Types {\rm (1)}--{\rm (4)} are foliated canonical, and Type
{\rm (5)} is foliated log canonical.

\item[\rm (ii)] In Type {\rm (2)}, either \(p\in Y\) is Du Val and the
singularity is not \(\epsilon\)-adjoint terminal, or the contraction of the chain
is smooth on the surface and the induced foliation is canonical (but not reduced).

\item[\rm (iii)] Types {\rm (3)} and {\rm (4)} are Du Val surface singularities
and are not \(\epsilon\)-adjoint terminal.
\end{enumerate}

\end{theorem}

\begin{proof}
An \(\epsilon\)-adjoint canonical singularity is \(\epsilon\)-adjoint log
canonical, so Theorem~\ref{thm:adjoint-lc-classification} applies.  The excluded
configurations are removed by
Lemma~\ref{lem:canonical-excluded-configurations}.  The remaining configurations
are exactly the five types listed above, with the additional restrictions supplied
by Lemma~\ref{lem:canonical-surviving-configurations}.

Types {\rm (1)}--{\rm (4)} are foliated canonical, and Type {\rm (5)} is
foliated log canonical, by the foliated log canonical classification recalled in
Theorem~\ref{thm:foliated-lc-classification}.
\end{proof}

\begin{corollary}[Canonical-to-log stability below \(1/4\)]
\label{cor:adjoint-canonical-stability}
Fix $\epsilon \in (0,1/4)$. Suppose $p$ is an $\epsilon$-adjoint canonical singularity of $(Y,\cG)$.
Then 
$p$ is log canonical as a foliated singularity and klt as a surface singularity.
\end{corollary}
\begin{proof}
It follows directly from Theorem~\ref{thm:adjoint-canonical-classification}, together with the classification of lc surface singularities (Theorem~\ref{thm:lc-surface}).
\end{proof}

\begin{remark}\label{remk:1/4sharp}
The bound \(1/4\) is sharp: at \(\epsilon=1/4\), a boundary configuration of
Type {\rm (III)} can be \(\epsilon\)-adjoint canonical but not foliated log
canonical; see Example~\ref{ex:e=1/4}.
\end{remark}

\subsection{Applications to the negative part and the adjoint MMP}
\label{subsec:canonical-applications}
We also record the corresponding form of the negative part in the Zariski
decomposition.

\begin{proposition}[Adjoint negative part below \(1/4\)]
\label{prop:adjoint-negative-part}
Fix $\epsilon\in(0,1/4)$.
Let \((X,\mathcal F)\) be a relatively minimal foliated surface (cf.~Definition~\ref{def:relatively-minimal}).  
Assume that \(K_{\mathcal F}+\epsilon K_X\) is pseudo-effective, and write its Zariski
decomposition as
\[
K_{\mathcal F}+\epsilon K_X=P_\epsilon+N_\epsilon.
\]
Then every connected component of \(\operatorname{Supp}N_\epsilon\) is one of
the following:
\begin{enumerate}
\item an \(\mathcal F\)-chain (cf.~Definition~\ref{def:F-chain});
\item a non-trivial \(\mathcal F\)-star chain (cf.~Definition~\ref{def:F-star-graph}) centered at a smooth rational
non-\(\mathcal F\)-invariant curve \(C\) satisfying
\(\operatorname{tang}(\mathcal F,C)=0\)
and \(C^2=-1\).
\end{enumerate}
\end{proposition}

\begin{proof}
Each connected component of \(\operatorname{Supp}N_\epsilon\)
is a \((K_{\mathcal F}+\epsilon K_X)^{>0}\)-graph, hence gives an \(\epsilon\)-adjoint terminal contraction.
By Theorem~\ref{thm:adjoint-canonical-classification} and
Lemma~\ref{lem:canonical-surviving-configurations}, the terminal possibilities
below \(1/4\) are the \(\mathcal F\)-chain case, the terminal subcase of Type
{\rm (2)}, and the non-trivial \(\mathcal F\)-star chain of Type {\rm (5)}.

The terminal subcase of Type {\rm (2)} contracts to a smooth surface point with
canonical induced foliation; equivalently, its exceptional chain contains an
\(\mathcal F\)-invariant \((-1)\)-curve \(L\) with
\(K_{\mathcal F}\cdot L\le0\).
This is excluded by minimality.  Types {\rm (3)} and {\rm (4)} are not
\(\epsilon\)-adjoint terminal.  Hence only the two listed types remain.
\end{proof}

For \(\epsilon=0\), the corresponding description of the negative part of
\(K_{\cF}\) is due to McQuillan; see \cite{MMcQ08} and the exposition in
\cite[Chapter~8]{Bru15}.  Proposition~\ref{prop:adjoint-negative-part} gives its
adjoint analogue for \(K_{\cF}+\epsilon K_X\) below the wall \(1/4\).
Thus the adjoint perturbation preserves the classical \(\mathcal F\)-chain blocks
and adds only the star-chain blocks in Proposition~\ref{prop:adjoint-negative-part}
{\rm (2)}.

The second possibility is not excluded uniformly in \(\epsilon\); see
Example~\ref{ex:epsilon-ad-can-NOT-can}.  However, for a fixed index of the
induced foliation, it disappears when \(\epsilon\) is sufficiently small.  For
instance, in the big case one may take the effective bound given in
\cite[Proposition~1.9]{LWX25}.

\begin{corollary}[Adjoint MMP below \(1/4\)]
\label{cor:adjoint-MMP}
Let $X$ be a smooth projective surface and $\cF$ be a foliation with canonical singularities.
Then for any $\epsilon\in(0,1/4)$, there exists a birational morphism 
$\varphi:X\to Y$ such that either
\begin{enumerate}
\item $K_{\cG}+\epsilon K_Y$ is nef, where $\cG=\varphi_*\cF$; or
\item there exists a morphism
$f:Y\to Z$
with \(\rho(Y/Z)=1\) such that
$-\bigl(K_{\cG}+\epsilon K_Y\bigr)$
is \(f\)-ample.
\end{enumerate}
Moreover, \(Y\) has klt singularities and \(\cG\) has log canonical
singularities.
\end{corollary}
\begin{proof}
Construct \(\varphi\) in two steps. First contract all \(\mathcal F\)-invariant
\((-1)\)-curves \(L\) with \(K_{\mathcal F}\cdot L\le 0\), obtaining a relatively
minimal model \((X',\mathcal F')\). Then contract the connected components of
the negative part of \(K_{\mathcal F'}+\epsilon K_{X'}\) as in Proposition~\ref{prop:adjoint-negative-part}.

If $K_{\cF}+\epsilon K_X$ is pseudo-effective, then $K_{\cG}+\epsilon K_Y$ is nef by Proposition~\ref{prop:adjoint-negative-part}.
If $K_{\cF}+\epsilon K_X$ is not pseudo-effective, then there is a curve \(C\subset Y\)
such that
\[
(K_{\cG}+\epsilon K_Y)\cdot C<0,
\qquad
C^2\ge0.
\]
By \cite[Corollary~3.17]{KM98} and \cite[Theorem~2.8]{SS23}, this gives a
contraction
\(
f:Y\to Z
\)
with \(\rho(Y/Z)=1\) such that
\(
-\bigl(K_{\cG}+\epsilon K_Y\bigr)
\)
is \(f\)-ample.
\end{proof}

\begin{corollary}[Adjoint canonical model below \(1/4\)]
\label{coro:adjoint-canonical-model}

With the notation of Corollary~\ref{cor:adjoint-MMP}, assume in addition that
$K_{\cF}+\epsilon K_X$ is big.  Then there exists a birational contraction
\[
h:(Y,\cG)\longrightarrow (Z,\cH)
\]
such that \((Z,\cH)\) is the adjoint canonical model.  Moreover, \(Z\) has klt
singularities, \(\cH\) has log canonical singularities, and
\(
K_{\cH}+\epsilon K_Z
\)
is ample.
\end{corollary}

\begin{proof}
Since \(K_{\mathcal F}+\epsilon K_X\) is big, the nef divisor
\(K_{\mathcal G}+\epsilon K_Y\) obtained in Corollary~\ref{cor:adjoint-MMP} is nef and big.
Contracting its null locus gives
a birational morphism \(h:(Y,\mathcal G)\to (Z,\mathcal H)\).
The singularities produced by this contraction are
\(\epsilon\)-adjoint canonical. Hence
Theorem~\ref{thm:adjoint-canonical-classification} gives that \(Z\) is klt and that
\(\cH\) is log canonical.
On \(Z\), the divisor \(K_{\cH}+\epsilon K_Z\) is nef and big.  Since \(Z\) has
klt, hence rational, singularities, \cite[Theorem~4.8]{Tan04} implies that
\(K_{\cH}+\epsilon K_Z\) is ample.
\end{proof}

In fact, by \cite{Tan04}, there is an explicit integer
\(\alpha=\alpha(K_{\cH}+\epsilon K_Z)\) such that
\(\alpha\cdot (K_{\cH}+\epsilon K_Z)\) is very ample; see also
\cite[Corollary~1.2]{LWX25}.
\begin{remark}
Corollaries~\ref{cor:adjoint-MMP} and~\ref{coro:adjoint-canonical-model} extend
\cite[Theorems~1.1 and~1.2]{SS23}, where the range
\(\epsilon\in(0,1/5)\) was treated.  They also complement
\cite[Theorem~1.1]{LWX25}.
\end{remark}

\section{Examples and sharpness}\label{sec:examples}

We give two examples illustrating the necessity of the new adjoint
configurations and the sharpness of the thresholds.  The first realizes the
non-trivial \(\mathcal F\)-star chain appearing in
Theorem~\ref{thm:adjoint-canonical-classification} and in
Proposition~\ref{prop:adjoint-negative-part}; the second realizes the boundary
configuration of Type {\rm (III)} in
Theorem~\ref{thm:adjoint-lc-classification}, and hence the walls \(1/5\) and
\(1/4\).

\begin{example}\label{ex:epsilon-ad-can-NOT-can}
Let $p$ be a singularity of a foliated surface $(X,\cF)$.
Assume that $X$ is smooth at $p$, and that $\cF$ is locally defined near
$p=(0,0)$ by the $1$-form
\[
\omega = ny\,\mathrm{d}x - x\,\mathrm{d}y.
\]
Consider the minimal resolution
$\sigma \colon (X',\cF') \to (X,\cF)$
over $p$ (see Figure~\ref{fig:adcanNOTcan}).

\begin{figure}[htbp]
  \centering
  \def\svgwidth{\columnwidth}
  \scalebox{0.5}{
\begingroup%
  \makeatletter%
  \providecommand\color[2][]{%
    \errmessage{(Inkscape) Color is used for the text in Inkscape, but the package 'color.sty' is not loaded}%
    \renewcommand\color[2][]{}%
  }%
  \providecommand\transparent[1]{%
    \errmessage{(Inkscape) Transparency is used (non-zero) for the text in Inkscape, but the package 'transparent.sty' is not loaded}%
    \renewcommand\transparent[1]{}%
  }%
  \providecommand\rotatebox[2]{#2}%
  \newcommand*\fsize{\dimexpr\f@size pt\relax}%
  \newcommand*\lineheight[1]{\fontsize{\fsize}{#1\fsize}\selectfont}%
  \ifx\svgwidth\undefined%
    \setlength{\unitlength}{143.13019122bp}%
    \ifx\svgscale\undefined%
      \relax%
    \else%
      \setlength{\unitlength}{\unitlength * \real{\svgscale}}%
    \fi%
  \else%
    \setlength{\unitlength}{\svgwidth}%
  \fi%
  \global\let\svgwidth\undefined%
  \global\let\svgscale\undefined%
  \makeatother%
  \begin{picture}(1,0.10751517)%
    \lineheight{1}%
    \setlength\tabcolsep{0pt}%
    \put(0.49462423,0.02317965){\color[rgb]{0,0,0}\makebox(0,0)[lt]{\lineheight{1.25}\smash{\begin{tabular}[t]{l}$\cdots\cdots$\end{tabular}}}}%
    \put(0,0){\includegraphics[width=\unitlength,page=1]{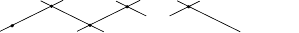}}%
    \put(0.22227328,0.06642696){\color[rgb]{0,0,0}\makebox(0,0)[lt]{\lineheight{1.25}\smash{\begin{tabular}[t]{l}$\bar{E}_2(-2)$\end{tabular}}}}%
    \put(0.07807406,0.02115494){\color[rgb]{0,0,0}\makebox(0,0)[lt]{\lineheight{1.25}\smash{\begin{tabular}[t]{l}$\bar{E}_1(-2)$\end{tabular}}}}%
    \put(0.69511248,0.05972004){\color[rgb]{0,0,0}\makebox(0,0)[lt]{\lineheight{1.25}\smash{\begin{tabular}[t]{l}$\bar{E}_{n-1}(-2)$\end{tabular}}}}%
    \put(0,0){\includegraphics[width=\unitlength,page=2]{adcanNOTcan.pdf}}%
    \put(0.85679886,0.04162384){\color[rgb]{0,0,0}\makebox(0,0)[lt]{\lineheight{1.25}\smash{\begin{tabular}[t]{l}$E_n(-1)$\end{tabular}}}}%
  \end{picture}%
\endgroup%
}
  \caption{Resolution graph in Example~\ref{ex:epsilon-ad-can-NOT-can}.}
  \label{fig:adcanNOTcan}
\end{figure}

Denote by $q_i$ (resp.\ $E_i$), $i=1,\dots,n$, the successive blow-up points
(resp.\ exceptional curves). In this case, we have
\[
l(q_i) = a(q_i) = 1, \quad i \le n-1, \qquad\text{and}\qquad l(q_n) = 2 > a(q_n) = 1.
\]
In particular, each $E_i$ is $\cF'$-invariant for $i \le n-1$, while $E_n$ is not $\cF'$-invariant.
A straightforward computation yields
\[
K_{\cF'}= \sigma^*(K_{\cF}) - E_n, \qquad
K_{X'}= \sigma^*(K_X) + \sum_{i=1}^{n-1}i\bar{E}_i + n E_n.
\]
Hence,
\[
K_{\cF'} + \epsilon K_{X'} = \sigma^*(K_{\cF} + \epsilon K_X) + \epsilon \sum_{i=1}^{n-1} i \bar{E}_i + (n \epsilon - 1) E_n.
\]

Thus, for every fixed \(\epsilon>0\), choosing \(n\) sufficiently large gives an
\(\epsilon\)-adjoint canonical singularity which is not foliated canonical.
\end{example}

\begin{remark}
Example~\ref{ex:epsilon-ad-can-NOT-can} corresponds to Type (5) in Theorem~\ref{thm:adjoint-canonical-classification}, 
which is the only configuration in which an $\epsilon$-adjoint canonical singularity fails to be foliated canonical for sufficiently small $\epsilon$.
\end{remark}

\begin{example}\label{ex:e=1/4}
Let $p$ be a singularity of a foliated surface $(X,\cF)$.
Assume that $X$ is smooth at $p$, and that $\cF$ is locally defined near
$p=(0,0)$ by the $1$-form
\[
\omega = x\,\mathrm{d}x + y^2\,\mathrm{d}y.
\]

Consider the minimal resolution
$\sigma \colon (X',\cF') \to (X,\cF)$
over $p$ (see Figure~\ref{fig:1/4KX}).

\begin{figure}[htbp]
  \centering
  \def\svgwidth{\columnwidth}
  \scalebox{0.4}{
\begingroup%
  \makeatletter%
  \providecommand\color[2][]{%
    \errmessage{(Inkscape) Color is used for the text in Inkscape, but the package 'color.sty' is not loaded}%
    \renewcommand\color[2][]{}%
  }%
  \providecommand\transparent[1]{%
    \errmessage{(Inkscape) Transparency is used (non-zero) for the text in Inkscape, but the package 'transparent.sty' is not loaded}%
    \renewcommand\transparent[1]{}%
  }%
  \providecommand\rotatebox[2]{#2}%
  \newcommand*\fsize{\dimexpr\f@size pt\relax}%
  \newcommand*\lineheight[1]{\fontsize{\fsize}{#1\fsize}\selectfont}%
  \ifx\svgwidth\undefined%
    \setlength{\unitlength}{86.31763311bp}%
    \ifx\svgscale\undefined%
      \relax%
    \else%
      \setlength{\unitlength}{\unitlength * \real{\svgscale}}%
    \fi%
  \else%
    \setlength{\unitlength}{\svgwidth}%
  \fi%
  \global\let\svgwidth\undefined%
  \global\let\svgscale\undefined%
  \makeatother%
  \begin{picture}(1,0.23187348)%
    \lineheight{1}%
    \setlength\tabcolsep{0pt}%
    \put(0.09522554,0.03615831){\color[rgb]{0,0,0}\makebox(0,0)[lt]{\lineheight{1.25}\smash{\begin{tabular}[t]{l}$\bar{E}_1(-3)$\end{tabular}}}}%
    \put(0,0){\includegraphics[width=\unitlength,page=1]{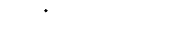}}%
    \put(0.60605719,0.10288551){\color[rgb]{0,0,0}\makebox(0,0)[lt]{\lineheight{1.25}\smash{\begin{tabular}[t]{l}$\bar{E}_2(-2)$\end{tabular}}}}%
    \put(0.34192596,0.12790865){\color[rgb]{0,0,0}\makebox(0,0)[lt]{\lineheight{1.25}\smash{\begin{tabular}[t]{l}$E_3(-1)$\end{tabular}}}}%
    \put(0,0){\includegraphics[width=\unitlength,page=2]{eK4.pdf}}%
  \end{picture}%
\endgroup%
}
  \caption{Resolution graph of the boundary example at the walls \(1/5\) and \(1/4\).}
  \label{fig:1/4KX}
\end{figure}

Denote by $q_i$ (resp.\ $E_i$) the successive blow-up points (resp.\ exceptional curves). 
In this case,
$l(q_1) = l(q_2) = 1$ and $l(q_3) = a(q_3) = 2$.
In particular, all $E_i$ are $\cF'$-invariant.
A direct computation gives
\[
K_{\cF'}= \sigma^*(K_{\cF}) - E_3, \qquad
K_{X'}= \sigma^*(K_X) + \bar{E}_1 + 2 \bar{E}_2 + 4 E_3.
\]
Hence,
\begin{align*}
K_{\cF'} + \frac{1}{4} K_{X'} 
&= \sigma^*\!\left(K_{\cF} + \frac{1}{4} K_X\right) + \frac{1}{4} \bar{E}_1 + \frac{1}{2} \bar{E}_2, \\
K_{\cF'} + \frac{1}{5} K_{X'} 
&= \sigma^*\!\left(K_{\cF} + \frac{1}{5} K_X\right) + \frac{1}{5} \bar{E}_1 + \frac{2}{5} \bar{E}_2 - \frac{1}{5} E_3.
\end{align*}

It follows that $p$ is $\tfrac{1}{4}$-adjoint canonical and $\tfrac{1}{5}$-adjoint log canonical, 
but it is not a log canonical singularity of $\cF$.
\end{example}

\begin{remark}
This example realizes both walls: it is \(\frac15\)-adjoint log canonical and
\(\frac14\)-adjoint canonical, but the foliation is not log canonical.
In fact, it corresponds to Type {\rm(III\mbox{-}1)} in Theorem~\ref{thm:adjoint-lc-classification},
i.e., the boundary configuration that is not foliated log canonical.
\end{remark}

\bigskip

\subsection*{Acknowledgements}
The author sincerely thanks Professors Shengli Tan, Jun Lu, and Xin L\"u for their continuous encouragement and support, and Professor Xiaohang Wu for many helpful discussions. 
He is also grateful to Professor Jihao Liu for raising the question concerning the sharp stability threshold in the adjoint log canonical setting, which inspired the development of this work.


\begin{thebibliography}{BHPV04}
  
\bibitem[Alex92]{Alex92}
V.~Alexeev,
\emph{Classification of log-canonical surface singularities: arithmetical proof},
in: Flips and Abundance for Algebraic Threefolds,
Papers from the Second Summer Seminar on Algebraic Geometry
held at the University of Utah, Salt Lake City, Utah, August 1991,
Ast\'erisque \textbf{211} (1992), 47--58.

\bibitem[Bad01]{Luc01}
L.~B\u{a}descu,
\emph{Algebraic Surfaces},
Universitext, Springer, New York, 2001.

\bibitem[BHPV04]{BPV04}
W.~Barth, K.~Hulek, C.~Peters, and A.~Van de Ven,
\emph{Compact Complex Surfaces},
2nd ed., Springer, Berlin, 2004.

\bibitem[Bru15]{Bru15}
M.~Brunella,
\emph{Birational Geometry of Foliations},
IMPA Monographs \textbf{1}, Springer, Cham, 2015.

\bibitem[CHL$^+$24]{CHL+24}
P.~Cascini, J.~Han, J.~Liu, F.~Meng, C.~Spicer, R.~Svaldi, and L.~Xie,
\emph{Minimal model program for algebraically integrable adjoint foliated structures},
arXiv:2408.14258 [math.AG], 2024.

\bibitem[Chen23]{YAChen23}
Y.-A.~Chen,
\emph{Log canonical foliation singularities on surfaces},
Math. Nachr. \textbf{296} (2023), 3222--3256.

\bibitem[KM98]{KM98}
J.~Koll\'ar and S.~Mori,
\emph{Birational Geometry of Algebraic Varieties},
Cambridge Tracts in Mathematics \textbf{134},
Cambridge University Press, Cambridge, 1998.


\bibitem[Lu25]{Lu25}
X. Lu,
\emph{Unboundedness of foliated varieties},
Internat. J. Math. \textbf{36} (2025), no.~6, Paper No.~2550003, 9~pp.


\bibitem[LWX25]{LWX25}
J.~Lu, X.-H.~Wu, and S.~Xu,
\emph{Canonical Models of Adjoint Foliated Structures on Surfaces},
arXiv:2501.00470v6 [math.AG], 2025.

\bibitem[McK22]{McK22}
J.~McKernan,
\emph{Log canonical thresholds for foliations},
JAMI 2022: Higher Dimensional Algebraic Geometry
(May 3--8, 2022), Johns Hopkins University, Baltimore.
An event in honor of Prof.~Shokurov's 70th Birthday.
Available at: \url{https://www.youtube.com/watch?v=ukF-oqJuypY}.

\bibitem[McQ08]{MMcQ08}
M.~McQuillan,
\emph{Canonical models of foliations},
Pure Appl. Math. Q. \textbf{4} (2008), no.~3, 877--1012.

\bibitem[PS19]{PS19}
J.~V.~Pereira and R.~Svaldi,
\emph{Effective algebraic integration in bounded genus},
Algebr. Geom. \textbf{6} (2019), no.~4, 454--485.

\bibitem[Sak84]{Sak84}
F.~Sakai,
\emph{Anticanonical models of rational surfaces},
Math. Ann. \textbf{269} (1984), 389--410.

\bibitem[Sei68]{Sei68}
A.~Seidenberg,
\emph{Reduction of singularities of the differential equation \(Ady=Bdx\)},
Amer. J. Math. \textbf{90} (1968), 248--269.

\bibitem[SS23]{SS23}
C.~Spicer and R.~Svaldi,
\emph{Effective generation for foliated surfaces: results and applications},
J. Reine Angew. Math. \textbf{795} (2023), 45--84.

\bibitem[Tan04]{Tan04}
S.-L.~Tan,
\emph{Effective behavior of multiple linear systems},
Asian J. Math. \textbf{8} (2004), no.~2, 287--304.

\bibitem[Vas25]{Vas25}
S.~Vassiliadis,
\emph{Explicit bounds on foliated surfaces and the Poincar\'e problem},
arXiv:2511.08388 [math.AG], 2025.

\end{thebibliography}
\end{document}